\documentclass[]{article}

\usepackage[fleqn]{amsmath}
\usepackage{amssymb,amsmath}
\usepackage{enumerate}
\usepackage{hyperref}
\usepackage{amsthm}
\usepackage{verbatim}
\usepackage[a4paper, total={13.5cm, 23.5cm}]{geometry}
\usepackage{fixfoot}
\usepackage{ centernot}
\usepackage[utf8]{inputenc}
\usepackage[T1]{fontenc}
\usepackage[parfill]{parskip}
\usepackage{graphicx}
\usepackage[dvipsnames]{xcolor}
\usepackage{bbm}
\usepackage{tabu}
\usepackage{tikz-qtree,tikz-qtree-compat}
\usepackage{todonotes}
\usepackage{csquotes}
\usepackage{enumitem}
\usepackage{tikz, tkz-euclide}
\usetikzlibrary{automata, math, matrix}
\usepackage{calc}
\usepackage{float}
\usepackage{pgfplots}
\usepackage{relsize}
\usepackage{tkz-euclide}
\usepackage{mathtools}
\usepackage{mathrsfs}
\usepackage{amsthm}
\usepackage{thmtools}
\usepackage[titletoc]{appendix}
\usepackage{chngcntr}
\usepackage{abstract}
\usepackage{glossaries}
\usepackage{longtable}
\usepackage{booktabs}
\usepackage{ltablex}
\usepackage{ latexsym }
\usepackage{wrapfig}
\usepackage{accents}
\usepackage{bm}
\usepackage{pgfplots}

\newcommand{\ind}{\mathbbm{1}}

\newcommand{\p}[1]{\left(#1\right)}

\newcommand{\abs}[1]{\left\vert#1\right\vert}
\def\Z{\mathbb{Z}}
\def\R{\mathbb{R}}
\def\sign{\textup{sign}}
\def\eps{\varepsilon}

\newtheorem{thrm}{Theorem}

\newtheorem{lemma}[thrm]{Lemma}
\newtheorem{claim}{Claim}

\newtheorem{prop}[thrm]{Proposition}

\newtheorem{conj}{Conjecture}
\newtheorem{problem}[conj]{Problem}

\begin{document}

\title{Multi-colour competition with reinforcement}
\date{May 31, 2022}
\author{Daniel Ahlberg and Carolina Fransson}
\maketitle

\begin{abstract}
We study a system of interacting urns where balls of different colour/type compete for their survival, and annihilate upon contact. For competition between two types, the underlying graph (finite and connected), determining the interaction between the urns, is known to be irrelevant for the possibility of coexistence, whereas for $K\ge3$ types the structure of the graph does affect the possibility of coexistence. We show that when the underlying graph is a cycle, competition between $K\ge3$ types almost surely has a single survivor, thus establishing a conjecture of Griffiths, Janson, Morris and the first author. Along the way, we give a detailed description of an auto-annihilative process on the cycle, which can be perceived as an expression of the geometry of a M\"obius strip in a discrete setting.
\end{abstract}

\section{Introduction}

Probabilistic problems phrased in terms of balls drawn from urns date back to Jacob Bernoulli's \emph{Ars Conjectandi} in 1713.
The diverse assortment of disciplines, in which central phenomena can be understood through an urn problem, have contributed to their continued importance.
P\'olya's urn model~\cite{eggpol23,polya30}, introduced around a century ago, is an archetype of a random process with reinforcement. Random processes with reinforcement are examples of processes where the entire trajectory of the process contributes to its eventual fate.
In P\'olya's model, a finite number of red and blue balls are initially placed in an urn. In each step of the process, a ball is drawn uniformly at random from the urn, and returned along with an identical copy of itself. The drawn colour is thus reinforced, in that balls of the same colour are more likely to be drawn in future steps. The effect of the reinforcement declines over time, resulting in the early steps of the process having a lasting effect over its continued evolution, and the limiting proportion of red balls in P\'olya's urn being random.

Random processes with reinforcement arise in a varied range of contexts, including economics, statistics and evolutionary biology; see~\cite{pemantle07} for a comprehensive survey. 
The long-term coordination observed in reinforcement processes have in different disciplines been featured under different names. In economics, such behaviour has been described as a \emph{lock-in} behaviour~\cite{arthur88}, whereas in physical contexts, the process is said to exhibit \emph{self-organisation}~\cite{baktanwie87,marval98}.
Systems of interacting urns form an interesting class of reinforcement processes, and have been used to model the formation of social networks~\cite{skypem00}, competition between business brands~\cite{benbenchelim15}, neuronal processing in the brain~\cite{hofholkuzrus16} and synchronisation in complex physical networks~\cite{alecrighi17}. 
The interaction between urns in the system can be determined by an underlaying discrete structure, which is thought to describe relations between agents, or to add a spatial dimension to the process. A compelling problem is to understand under what circumstances the spatial component will manifest in the lock-in or self-organisation of the process.

We shall in this paper be concerned with a model of interacting urns where balls of different type annihilate upon contact. Previous work~\cite{ahlgrijanmor19} have revealed that the spatial component, in this setting, has an effect on the eventual fate of the process, as we describe further below. The annihilative feature of the process has in the physics literature been suggested as a model for the inert chemical reaction $A+B\to\emptyset$; see~\cite{ovczel78,touwil83} and studied further in~\cite{braleb91a,braleb91b}. However, let us mention that our motivation originates from~\cite{ahlgrijanmor19}, where the annihilative feature was found to capture competition between different types in a planar growth model with reinforcement.

Consider the graph $C_N$ consisting of $N$ vertices and $N$ edges positioned in a cyclic fashion. Imagine an urn positioned at each vertex of the cycle. Place balls of $K\ge2$ different types into the urns so that no urn contains balls of more than one type. The balls in the system (either initially present or added later) are equipped with independent Poisson clocks. At the ring of a clock, the corresponding ball sends a copy of itself to each neighbour of the graph. Balls of different type are not allowed in the same urn simultaneously, and when a ball is added to an urn which contains balls of a different type, then the new ball annihilates with one of the already existing balls. We shall refer to this model as the \emph{$K$-type competing urn scheme} on the cycle.

The above competing urn scheme was introduced in~\cite{ahlgrijanmor19}, where the authors showed that for $K=2$, irregardless of the initial configuration of balls, the system of urns will eventually almost surely consist of balls of a single type. In fact, the same conclusion was shown to holds for the corresponding urn scheme on \emph{any} finite connected graph. The authors conjectured that also for $K\ge3$, in the case of the cycle, coexistence of $K$ types has probability zero. The authors further illustrated, by means of an example, that the case $K\ge3$ is strikingly different in that there are finite connected graphs for which mutual coexistence of three or more types \emph{does} occur with positive probability.

In this paper we settle the aforementioned conjecture from~\cite{ahlgrijanmor19}.

\begin{thrm}\label{thm:urn}
For every $N\ge K\ge 2$, and any nonzero initial configuration, the $K$-type competing urn scheme on the cycle of length $N$ has almost surely a single surviving type.
\end{thrm}

Models used to describe competing technologies have in the economics literature come to incorporate reinforcement mechanisms~\cite{arthur88}. In more physically motivated models for competing growth, reinforcement is often absent. This includes multi-type versions of the models of Eden~\cite{eden61} and Richardson~\cite{richardson73}. These models exhibit spatial coexistence among multiple competitors of equal strength~\cite{deihag08,hagpem98}.
Motivation for the competing urns scheme introduced in~\cite{ahlgrijanmor19} came form competition in a version of the Eden model on $\Z^2$ with a bootstrapping effect similar to that of bootstrap percolation and related automata~\cite{morris17}. The bootstrapping mechanism can be thought of as a spatial analogue of a reinforcement effect. In fact, it turns out that the question of coexistence in this competing growth model on $\Z^2$ can be rephrased in terms of coexistence in the competing urn scheme on a cycle.
By ruling out coexistence among two types on cycles of any length, the authors of~\cite{ahlgrijanmor19} were able to rule out coexistence among two types for the growth model on $\Z^2$. As a corollary of Theorem~\ref{thm:urn} we extend their result for the growth model on $\Z^2$ to any number $K\ge2$ types, hence proving another conjecture from~\cite{ahlgrijanmor19}; see Section~\ref{sec:lattice} below.
These results show how the presence of reinforcement may break the ties that arise in models for competing growth on $\Z^2$.

The study of (generalised) P\'olya urns via embedding in multi-type branching processes was pioneered in~\cite{athkar68}. Leaning on martingale analysis and Perron-Frobenius theory, they could determine the long-term behaviour of certain urn processes.
Also the two-type competing urn scheme in~\cite{ahlgrijanmor19} was analysed via the theory for multi-type branching processes. 
The methods required for its multi-type analogue are different in that the competing urn scheme, for $K\ge3$, will reduce to a process which, in general, is neither a branching process, nor amenable to Perron-Frobenius theory. Nevertheless, we shall see that it is sufficiently similar to a branching process in order for martingale methods to apply, which will reveal a more complex behaviour of the processes. We shall next, in Section~\ref{sec:AGJM}, outline the analysis of the two-type process, before we proceed to describe a process with an auto-adverse behaviour, resulting from the analysis of the multi-type process, in Section~\ref{sec:signed}.

\subsection{The two-type competing urn scheme}\label{sec:AGJM}

Let $G$ be a finite connected graph and $A$ its adjacency matrix. The two-type competing urn scheme on $G$ can be described as follows:
Encode type 1 balls as $1$s and type 2 balls as $-1$s. Let $Y(t)=(Y_k(t))_{k=1}^N$ be the vector that encodes the configuration of balls at time $t\ge0$. Then, if a nucleation occurs at time $t$ at position $k$, the change equals
\begin{equation}\label{eq:two-type}
Y(t)-Y(t-)=A_k\cdot\sign(Y_k(t-)),
\end{equation}
where $A_k$ is the $k$th column of the matrix $A$ and $\sign(\,\cdot\,)$ the sign function.

For initial configurations consisting of only one type, no annihilations will ever occur. In this case~\eqref{eq:two-type} describes a multi-type continuous-time Markov branching process with mean matrix given by $A$, and after proper rescaling the process converges almost surely to a random multiple of the Perron-Frobenius vector~\cite{athreya68}. By considering a conservative version of the two-type competing urn scheme, in which the interaction between balls of the different types result in a third type instead of annihilation, it was in~\cite{ahlgrijanmor19} possible to express the evolution of the competing urn scheme as the difference between two (dependent) multi-type branching processes. This was a key step in order to rule out coexistence in the case that $K=2$, resulting in the following theorem.

\begin{thrm}\label{thm:AGJM}
Let $G$ be a finite connected graph and let $\lambda$ denote the largest positive eigenvalue of the adjacency matrix of $G$ and $v$ the corresponding eigenvector. Then, for any nonzero initial configuration, there exists an almost surely nonzero random variable $W$ such that almost surely
$$
\lim_{t\to\infty}e^{-\lambda t}Y(t)=Wv.
$$
\end{thrm}

From Perron-Frobenius theory we know that the largest eigenvector has multiplicity one and that the corresponding eigenvector has strictly positive entries. In the case of the cycle $C_N$ we have $\lambda=2$ and the corresponding eigenvector $v$ is the all ones vector $(1,1,\ldots,1)$. By Theorem~\ref{thm:urn} it follows that all coordinates of $Y(t)$ will be either positive or negative for large values of $t$, almost surely, from which the case of $K=2$ of Theorem~\ref{thm:AGJM} follows immediately.

The main reason that a similar approach fails for $K\ge3$ is that it is not obvious how to encode the competition between three or more types in an arithmetic system. Indeed, such a system would have to have a non-Abelean feature in that the order in which balls are placed into the urns matters. That this feature has importance was illustrated in~\cite{ahlgrijanmor19} by giving examples of graphs for which coexistence among $K\ge3$ types occurs with positive probability.

\begin{figure}[h!]
	\centering
	\begin{tikzpicture}[scale=1.0]
	\draw (0:2) arc (0:360:20mm);
	\foreach \phi in {1,...,6}{
		\node[state,fill=white] (\phi) at (120-360/6 * \phi:2cm) {};
		\node[] (\phi) at (120-360/6 * \phi:2.8cm) {$v_\phi$};	
	}
	\node [] at (120-360/6 * 1:2.1cm)[circle, fill=gray,draw=black,inner sep=1.5pt] {};
	\node [] at (180-360/6 * 1:2.15cm)[circle,fill=gray,draw=black,inner sep=1.5pt] {};
	\node [] at (180-360/6 * 1:1.85cm)[circle,fill=gray,draw=black,inner sep=1.5pt] {};
	\node [] at (120-4-360/6 * 2:2.1cm)[circle,draw=black, inner sep=1.5pt] {};
	\node [] at (120+4-360/6 * 2:2.1cm)[circle,draw=black, inner sep=1.5pt] {};
\node [] at (120-360/6 * 2:1.9cm)[circle,draw=black, inner sep=1.5pt] {};
	\node [] at (120-4-360/6 * 3:2.1cm)[circle,draw=black, inner sep=1.5pt] {};
\node [] at (120+4-360/6 * 3:2.1cm)[circle,draw=black, inner sep=1.5pt] {};
\node [] at (120-360/6 * 3:1.9cm)[circle,draw=black, inner sep=1.5pt] {};
	\node [] at (180-360/6 * 5:2.15cm)[circle,draw=black,inner sep=1.5pt] {};
\node [] at (180-360/6 * 5:1.85cm)[circle,draw=black,inner sep=1.5pt] {};
\node [] at (180-360/6 * 6:1.9cm)[circle,draw=black,inner sep=1.5pt] {};
	\end{tikzpicture}
	\caption{Competition between two types on a cycle of length six. In the figure there are two tribes, consisting of the vertices $\{v_1,v_6\}$ and $\{v_2,v_3,v_4,v_5\}$ respectively, and two fronts between vertices $v_1$ and $v_2$, and $v_5$ and $v_6$.}\label{fig:mg_illu1}
\end{figure}
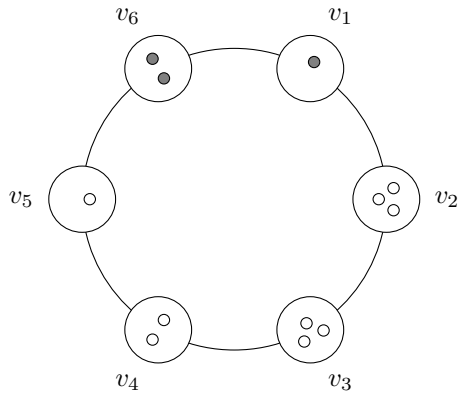

As observed in~\cite{ahlgrijanmor19}, this issue can in part be circumvented for competition on $C_N$, due to the structure of the graph. Given a configuration of $K\ge2$ types on $C_N$ we shall by the word \emph{tribe} refer to a maximal connected subgraph of $C_N$ for which each vertex is either vacant or occupied by balls of the same type.\footnote{Note that when there are no vacant vertices the collection of tribes forms a partition of the vertices of $C_N$, but that vacant vertices may belong to more than one tribe.} The separation between two tribes will be referred to as a \emph{front}; see Figure~\ref{fig:mg_illu1}.  The number of fronts is then equal to the number of tribes (as long as there are at least two tribes, and zero otherwise) and can be described as the number of occupied vertices of the cycle such that the closest occupied vertex in the clockwise direction is occupied by a different type. Note that balls on $C_N$ are ``aware'' that balls of a tribe with which they share a front are of a different type than their own, but ``oblivious'' to the same for balls in tribes that are separated by at least two fronts. Consequently, in the case that the initial configuration consists of an even number of tribes, it is possible to encode the competition between the $K$ types by alternatingly giving tribes positive or negative numbers. This encoding will no longer be an equivalent formulation of the process, as the encoding, in general, will no longer describe the competing urn process once a tribe is eliminated, and the balls clockwise and counter-clockwise of the eliminated tribe become ``aware'' of each other.

In the case that the number of tribes is odd, encoding the competition process in the above manner will result in the first and last tribes being encoded by positive numbers. Competition according to~\eqref{eq:two-type} will then not result in annihilation between balls in the first and final tribes, as it should. This can be adjusted for by declaring the edge (or one of the edges if several) connecting the first and final tribes as a ``sign-reversing'' edge, with the property that every ball that is sent along this edge has its sign reversed. This process will be properly introduced below, and studied to certain lengths.

\subsection{Auto-annihilative growth on a circular graph}\label{sec:signed}

Large parts of this paper will consist in analysing a growth process on the cycle with an auto-adverse behaviour. We shall next describe this process, the results of our analysis thereof, as well as its connection to the $K$-type competing urn scheme.

Consider the following process on the cycle of length $N$. At time zero, distribute a finite number of signed particles to the vertices of the cycle in a way such that no particles of opposing signs are at the same position. We may encode the initial configuration of particles by a vector in $\Z^N$. When time starts, all particles (initially present or born later) reproduce according to independent Poisson clocks. At the ring of a clock, the particle associated with the clock sends an independent copy of itself to each neighbour, except for particles sent across the edge connecting vertices $1$ and $N$, in which case the sent particle has its sign changed. Particles of the same sign do not interact, but particles of opposing signs that appear at the same vertex annihilate immediately on a one-by-one basis; see Figure~\ref{fig:mg_illu}.

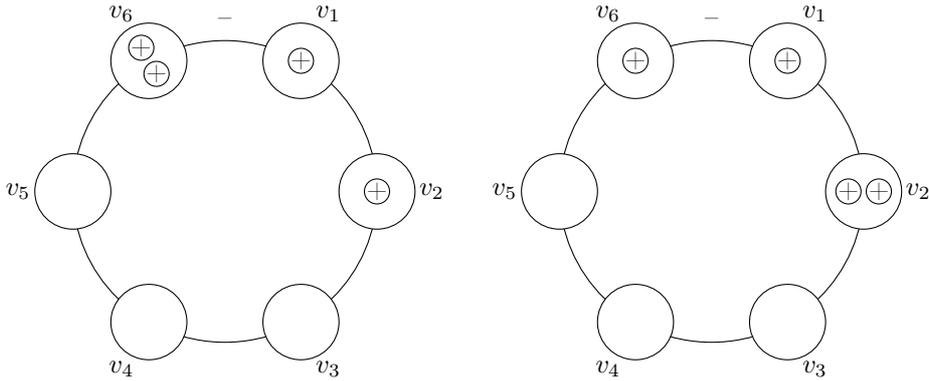
\begin{figure}[h!]
\centering
\begin{tikzpicture}[scale=.8]
\draw (0:2.5) arc (0:360:25mm);
\foreach \phi in {1,...,6}{
	\node[state,fill=white,  minimum size=1cm] (\phi) at (120-360/6 * \phi:2.5cm) {};
	\node[] (\phi) at (120-360/6 * \phi:3.4cm) {$v_\phi$};	
}

\node[] at (90:2.85cm) {--};
\node [draw, circle] at (120-360/6 * 1:2.5cm) {};
\node [] at (120-360/6 * 1:2.5cm) {+};

\node [] at (180-360/6 * 1:2.75cm) {+};
\node [draw, circle] at (180-360/6 * 1:2.75cm) {};

\node [] at (180-360/6 * 1:2.25cm) {+};
\node [draw, circle] at (180-360/6 * 1:2.25cm) {};

\node [] at (120-360/6 * 2:2.5cm) {+};
\node [draw, circle] at (120-360/6 * 2:2.5cm) {};

\begin{scope}[shift={(8,0)}]
\draw (0:2.5) arc (0:360:25mm);
\foreach \phi in {1,...,6}{
	\node[state,fill=white,  minimum size=1cm] (\phi) at (120-360/6 * \phi:2.5cm) {};
	\node[] (\phi) at (120-360/6 * \phi:3.4cm) {$v_\phi$};	
}

\node[] at (90:2.85cm) {--};
\node [draw, circle] at (120-360/6 * 1:2.5cm) {};
\node [] at (120-360/6 * 1:2.5cm) {+};

\node [] at (180-360/6 * 1:2.5cm) {+};
\node [draw, circle] at (180-360/6 * 1:2.5cm) {};

\node [] at (120-360/6 * 2:2.25cm) {+};
\node [draw, circle] at (120-360/6 * 2:2.25cm) {};

\node [] at (120-360/6 * 2:2.75cm) {+};
\node [draw, circle] at (120-360/6 * 2:2.75cm) {};
\end{scope}
\end{tikzpicture}
\caption{Illustration of the graph $C_6$ with the sign-reversing edge marked by a negation sign. If the clock of the particle at node $v_1$ rings (left), then a particle is added to each of the neighbouring urns at $v_2$ and $v_6$ (right). Since the particle sent across the sign-reversing edge is negated, the nucleation results in one annihilation.}\label{fig:mg_illu}
\end{figure}

Let $Z(t)=(Z_k(t))_{k=1}^N$ be the vector that encodes the configuration of particles at time $t$. We shall refer to $(Z(t))_{t\ge0}$ as the \emph{signed competition} or \emph{auto-annihilative growth process} on the cycle of length $N$. Let $A^*$ denote the matrix
\begin{gather}\label{eq:A}
A^*:=\begin{pmatrix}
0 & 1& 0&\cdots & 0 & -1 \\
1 & 0 & 1&\cdots & 0 & 0 \\
0& 1 & 0 &\cdots & 0 & 0 \\
\vdots& \vdots & \vdots & \ddots& \vdots&\vdots \\
0&  0 &0&\cdots& 0 & 1 \\
-1&  0 & 0&\cdots& 1& 0  \\
\end{pmatrix}.
\end{gather}
That is, $A^*$ is obtained from the adjacency matrix $A$ of $C_N$ by negating the entries in the SW and NE corners. Note that if at time $t$ a clock rings at position $k$, then the change caused by the ring can be expressed as
\begin{equation}\label{eq:multi-type}
Z(t)-Z(t-)=A^*_k\cdot\sign(Z_k(t-)),
\end{equation}
where $A^*_k$ denotes the $k$-th column of the matrix $A^*$. We shall henceforth refer to $A^*$ as the \emph{reinforcement matrix} of the auto-annihilative growth process.

The introduction of a sign-reversing edge has interesting effects on the geometry of the process. Imagine, in the auto-annihilative process, descendants of a particle spreading around the cycle. Once descendants of that particle has completed their first lap clockwise, their sign have been negated. Hence, even when started with a single particle, both positive and negative particles will (or at least may) eventually become present. Completing yet another lap clockwise, the descendants of the original particle return with their original sign, much like an ant on a M\"obius strip. This suggests that the geometry of the auto-annihilative process is encoded in the cycle of twice the length.

We emphasise the fact that~\eqref{eq:multi-type} does not describe a multi-type branching process (at least not of the standard form) as the mean/reinforcement matrix $A^*$ has negative entries. Moreover, Perron-Frobenius theory does not apply, and the largest eigenvalue of $A^*$ will no longer (necessarily) be unique.
Inspired by the discussion in the previous paragraph we determine, in Section~\ref{sec:eigenstructure}, the eigensystem of the reinforcement matrix $A^*$, and describe its relation to the adjacency matrix of a cycle of double the length. In particular, we show that the largest positive eigenvalue of the reinforcement matrix equals $2\cos(\pi/N)$ and has multiplicity two. The eigenspace associated to the largest eigenvalue is spanned by the two (orthogonal) vectors
\begin{equation}\label{eq:first_vectors}
v^{(1)}:=\big(\cos \p{\pi(j-1)/N}\big)_{j=1}^N\quad\text{and}\quad v^{(2)}:=\big(\sin\p{\pi(j-1)/N}\big)_{j=1}^N.
\end{equation}
These vectors are described by harmonic functions. Since a linear combination of harmonic functions (with the same period) is again an harmonic function, one may expect that $Z(t)$, for large $t$, should again be described by an harmonic curve. This is indeed the case, and is one of the main results of this paper.

\begin{thrm}\label{thm:signed}
Consider the auto-annihilative growth process on $C_N$ for $N\ge4$, and set $\lambda^*:=2\cos(\pi/N)$. Then, for any nonzero initial configuration, there exist continuous random variables $R$ and $S$, whose joint distribution is fully supported on $\R\times[-\pi/2,\pi/2)$, such that almost surely
$$
\lim_{t\to\infty}e^{-\lambda^* t}Z(t)=\big(R\cos\big(\pi(j-1)/N+S\big)\big)_{j=1}^N.
$$
\end{thrm}

The limiting behaviour of Theorem~\ref{thm:signed} is illustrated in Figure~\ref{fig:M1}.

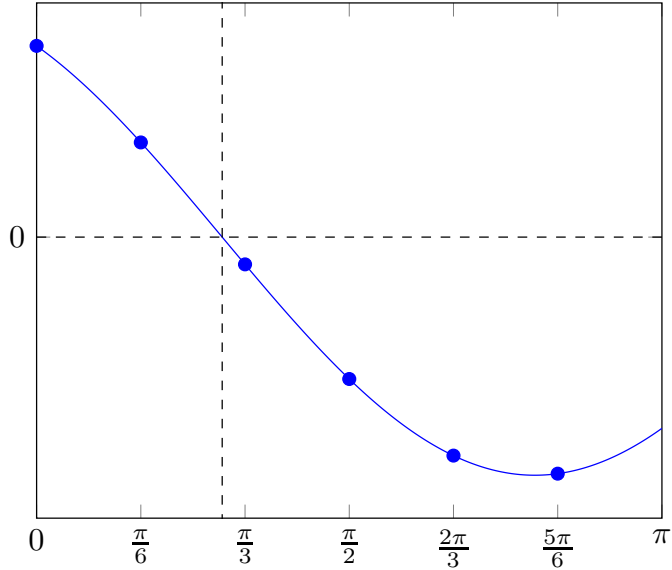
\begin{figure}
	\centering
\tikzmath{\S = 0.2032*pi; \lista={-pi*3/6, -pi*2/6,...,pi*2/6};}
 \begin{tikzpicture}[scale=1.2]
\begin{axis}[
xmin=0,xmax=pi,
ytick={0},
yticklabels={$0$},
xtick={ 0,  pi/6, 2*pi/6, pi/2, 2*pi/3,5*pi/6, pi},
xticklabels={ $0$, $\frac{\pi}{6}$, $\frac{ \pi}{3}$,$\frac{\pi}{2}$, $\frac{2 \pi}{3}$,$\frac{5 \pi}{6}$, $\pi$}
]
\addplot[domain=0:pi,samples=200,blue]{cos(deg(x+\S))};
\addplot[domain=0:pi,draw=none, samples at={0, pi/6,...,pi*5/6}, blue,mark=*]{cos(deg(x+\S))};
\draw[dashed](axis cs:-0,0) -- (axis cs:pi,0);
\draw[dashed](axis cs:pi/2-\S,1) -- (axis cs:pi/2 -\S,-1.5);
\end{axis}
\end{tikzpicture}
\caption{An illustration of the asymptotic composition of the auto-annihilative growth process on $C_6$, described by the curve $x\mapsto r\cos(x+s)$. In the limit the process consists of one tribe and one front, which here is located between $v_2$ and $v_3$. } \label{fig:M1}
\end{figure}

It is natural to talk about tribes and fronts also for the cycle with a sign-reversing edge. In this context, we define a \emph{tribe} to be a maximal connected subgraph of the cycle which either does not contain the sign-reversing edge and for which all vertices are either vacant or occupied by particles of the same sign, or which does contain the sign-reversing edge and where vertices clockwise and counterclockwise of that edge (which are not vacant) are occupied by particles of opposing sign. Again, we refer to the separation between two tribes as a \emph{front}. Note that in the auto-annihilative process the number of fronts coincides with the number of tribes, and that as long as the initial configuration is nonzero, there will always be at least one front present.

From Theorem~\ref{thm:signed} we deduce that, regardless of the number of tribes initially present, the system will eventually consist of a single tribe and a single front, where the process competes with itself. We note, moreover, that the location of the front corresponds to the location where the function $x\mapsto R\cos(x+S)$ has its unique zero on the interval $[0,\pi)$; the front is located between vertices $j$ and $j+1$ (vertex $N+1$ is the same as vertex $1$) if the zero occurs in the interval $\big((j-1)\pi/N,j\pi/N\big)$; see Figure~\ref{fig:M1}.
We emphasise the fact that the random variables $R$ and $S$ are continuous. While this will have marginal importance for the deduction of Theorem~\ref{thm:urn}, it is of interest as it shows that there is no specific amplitude nor shift to which the evolution of the process locks into.

Apart from being a key step toward a proof of Theorem~\ref{thm:urn}, the auto-annihilative process and Theorem~\ref{thm:signed} can be thought of as an attempt to understand the behaviour of competing branching random walks on the integers, evolving in ``equilibrium'', as we elaborate upon at the end of this paper.


\subsection{Competing growth on $\Z^2$}\label{sec:lattice}

The competing urn scheme, introduced in~\cite{ahlgrijanmor19}, arose from a model for competing growth on the $\Z^2$ nearest neighbour lattice. Lattice models for spatial growth were first studied by Eden~\cite{eden61}. In his model, a finite number of sites of $\Z^2$ are initially occupied, and all other sites are vacant. As time evolves, vacant sites with at least one occupied neighbour become occupied at rate 1. In~\cite{ahlgrijanmor19} the authors were concerned with a variant of the Eden model, in which vacant sites with one occupied neighbour become occupied at rate 1, whereas vacant sites with at least two occupied neighbours become occupied immediately. The Eden model and the variant of the Eden model from~\cite{ahlgrijanmor19} constitute two extreme point of a larger family of nucleation and growth processes introduced in~\cite{kessch95}, and further studied in~\cite{bolgrimorrolsmi20,cerman13a,dehsch97a}. The immediate occupation of sites with at least two occupied neighbours results in a bootstraping effect, similar to that of bootstrap percolation and related automata. A multi-type version of the corresponding model is defined analogously: Sites are initially either vacant or occupied by one of $K$ different types. Vacant sites with one occupied neighbour of type $j$ become occupied by type $j$ at rate 1, and vacant sites with at least two neighbours occupied by type $j$ become occupied by type $j$ immediately.

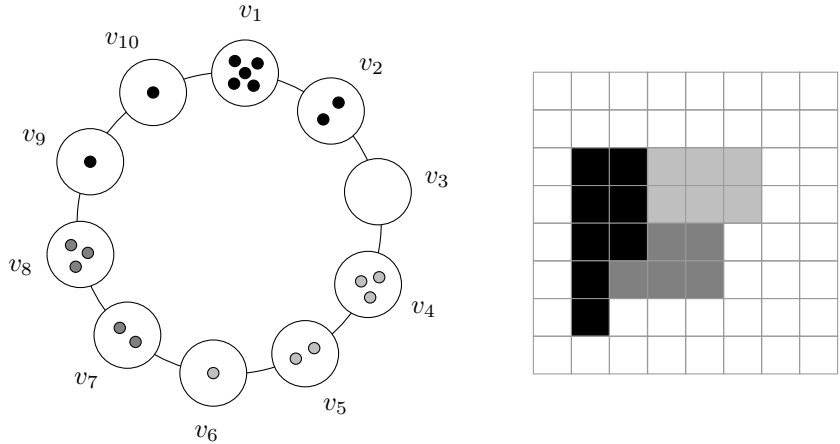
\begin{figure}[h!]
	\centering
	\begin{tikzpicture}[scale=1]
	\draw (0:2) arc (0:360:20mm);
	\foreach \phi in {1,...,10}{
		\node[state,fill=white] (\phi) at (120-360/10 * \phi:2cm) {};
		\node[] (\phi) at (120-360/10 * \phi:2.8cm) {$v_{\phi}$};	
	}
	\node [] at (120-360/10 * 1:2cm)[circle, fill=black,draw=black,inner sep=1.5pt] {};
	\node [] at (120-360/10+4 * 1:2.15cm)[circle, fill=black,draw=black,inner sep=1.5pt] {};
	\node [] at (120-360/10-4 * 1:2.15cm)[circle, fill=black,draw=black,inner sep=1.5pt] {};
	\node [] at (120-360/10+4 * 1:1.85cm)[circle, fill=black,draw=black,inner sep=1.5pt] {};
\node [] at (120-360/10-4 * 1:1.85cm)[circle, fill=black,draw=black,inner sep=1.5pt] {};
\node [] at (120-360/10 * 2:2.15cm)[circle, fill=black,draw=black,inner sep=1.5pt] {};
\node [] at (120-360/10 * 2:1.85cm)[circle, fill=black,draw=black,inner sep=1.5pt] {};

\node [] at (120-360/10 * 4:1.9cm)[circle, fill=gray!50,draw=black,inner sep=1.5pt] {};
\node [] at (120+4 -360/10* 4:2.1cm)[circle, fill=gray!50,draw=black,inner sep=1.5pt] {};
\node [] at (120-4-360/10 * 4:2.1cm)[circle, fill=gray!50,draw=black,inner sep=1.5pt] {};
\node [] at (120+4 -360/10* 5:2cm)[circle, fill=gray!50,draw=black,inner sep=1.5pt] {};
\node [] at (120-4-360/10 * 5:2cm)[circle, fill=gray!50,draw=black,inner sep=1.5pt] {};
\node [] at (120-360/10* 6:2cm)[circle, fill=gray!50,draw=black,inner sep=1.5pt] {};
\node [] at (120+4 -360/10* 7:2cm)[circle, fill=gray,draw=black,inner sep=1.5pt] {};
\node [] at (120-4-360/10 * 7:2cm)[circle, fill=gray,draw=black,inner sep=1.5pt] {};
\node [] at (120-360/10 * 8:1.9cm)[circle, fill=gray,draw=black,inner sep=1.5pt] {};
\node [] at (120+4 -360/10* 8:2.1cm)[circle, fill=gray,draw=black,inner sep=1.5pt] {};
\node [] at (120-4-360/10 * 8:2.1cm)[circle, fill=gray,draw=black,inner sep=1.5pt] {};
\node [] at (120-360/10 * 9:2cm)[circle, fill=black,draw=black,inner sep=1.5pt] {};
\node [] at (120-360/10 * 10:2cm)[circle, fill=black,draw=black,inner sep=1.5pt] {};

\begin{scope}[shift={(6,0)}]
\matrix[matrix of nodes, nodes={draw=black!40,minimum size=0.5cm}, nodes in empty cells,column sep=-\pgflinewidth,row sep=-\pgflinewidth](M){
	&              &              &                &                 &                &              &           \\
	&              &              &                &                 &                &              &           \\
	&|[fill=black]|&|[fill=black]|&|[fill=gray!50]|& |[fill=gray!50]|&|[fill=gray!50]|&              &           \\
	&|[fill=black]|&|[fill=black]|&|[fill=gray!50]|& |[fill=gray!50]|&|[fill=gray!50]|&              &           \\
	&|[fill=black]|&|[fill=black]|& |[fill=gray]|  &|[fill=gray]|    &                &              &           \\
	&|[fill=black]|& |[fill=gray]|& |[fill=gray]|  &|[fill=gray]|    &                &              &         \\
	&|[fill=black]|&              &                &                 &                &              &           \\
	&              &              &                &                 &                &              &           \\
};
\end{scope}
	\end{tikzpicture}
	\caption{An instance of the competing urn scheme on $C_{10}$ (left) and the competing growth process on $\Z^2$ with a bootstrapping effect (right). Starting at the SW corner, we may encode the perimeter of the occupied cluster on the right by the configuration of balls on the left. The growth of a side of the cluster on the right corresponds to the nucleation of a ball on the left.}\label{fig:mg_illu2}
\end{figure}

For the $K$-type competing growth model on $\Z^2$, with a finite number of initially occupied sites, we say that type $j$ \emph{survives} if infinitely many sites eventually become occupied by type $j$. In~\cite{ahlgrijanmor19} the authors showed that for this process consisting of two types, there is almost surely only a single surviving type. Moreover, the authors conjectured that the same holds also for $K\ge3$, which we here confirm.

\begin{thrm}\label{thm:plane}
For every $K\ge2$ and every finite initial configuration, the $K$-type competing growth model on $\Z^2$ has almost surely a single surviving type.
\end{thrm}

We shall, in Section~\ref{sec:plane}, deduce Theorem~\ref{thm:plane} from Theorem~\ref{thm:urn}. In brief, the result follows from the observation, from~\cite{ahlgrijanmor19}, that competition between $K$ types on $\Z^2$ can be encoded as a $K$-type competing urn scheme on a cycle of a certain length, by following the perimeter of the occupied cluster; see Figure~\ref{fig:mg_illu2}.

\subsection{Notation and outline of the paper}

We shall primarily think of vectors as rows, and multiply matrices by vectors from the left. However, for ease of notation, we shall occasionally think of the vector $Z(t)$ as a column vector, and write $vZ(t)$ for the (scalar) product of $v$ and $Z(t)$. For a vector $v$ in $\mathbb{R}^N$, we shall use the notation $|v|:=(\sum_{j=1}^Nv_j^2)^{1/2}$ and $\|v\|:=\sum_{k=1}^N|v_k|$ for its $\ell_2$- and $\ell_1$-norm, respectively.

We proceed, in Section~\ref{sec:eigenstructure}, to examine the eigensystem of the reinforcement matrix associated to the auto-annihilative growth process on $C_N$. Based on this analysis we define, in Section~\ref{sec:martingale}, a series of martingales that we show are convergent almost surely and in $L^2$. The martingale limits are analysed further in Section~\ref{sec:limit}, leading to a proof of Theorem~\ref{thm:signed}. Theorem~\ref{thm:urn} is then proved in Section~\ref{sec:proof}, and Theorem~\ref{thm:plane} is derived as a corollary in Section~\ref{sec:plane}. We end the paper by mentioning a few connections to open problems in Section~\ref{sec:open}.

\section{Eigenstructure of the reinforcement matrix}\label{sec:eigenstructure}

We will in this section describe the eigensystem of the $N\times N$-reinforcement matrix $A^*$, which will be relevant for our later analysis. We do this by first describing the adjacency matrix of $C_{2N}$, the cycle of length $2N$, and deduce properties of $A^*$ therefrom. We denote by $\tilde A$ the $2N\times2N$-matrix describing the adjacency structure of $C_{2N}$.

Both $A^*$ and $\tilde A$ are real-valued and symmetric, so their eigenvalues are real. The matrix $\tilde A$ is, in addition, a circulant matrix, which is a matrix in which each column is obtained by the shifting all elements of the previous column one step down, in a cyclic fashion.\footnote{Equivalently, every row is obtained by shifting the previous row one step right.} Hence, the matrix is completely determined by its first column. Circulant matrices have the notable property that they all have the same set of eigenvectors.\footnote{We shall here be concerned with left-eigenvectors and treat vectors as rows.}

Clearly, the all ones vector is a (left-)eigenvector of the matrix, since the elements of every column sum up to the same number. The remaining eigenvectors can be expressed in terms of roots of unity. Let $\omega:=e^{\pi i/N}$, and for $k=0,1,\ldots,2N-1$ set
$$
\tilde v^{(k)}:=\big(1,\omega^k,\omega^{2k},\ldots,\omega^{(2N-1)k}\big).
$$
Indeed, note that multiplying $\tilde A$ by $\tilde v^{(k)}$ from the left results in a vector whose $j$th element equals
$$
(\tilde v^{(k)}\tilde A)_j=\omega^{(j-2)k}+\omega^{jk}=[\omega^{-k}+\omega^k]\omega^{(j-1)k}=2\cos(\pi k/N)(\tilde v^{(k)})_j.
$$
Hence, $\tilde v^{(k)}$ is indeed a (left-)eigenvector of $\tilde A$ with eigenvalue $\tilde\lambda_k:=2\cos(\pi k/N)$.

We note that the largest positive eigenvalue of $\tilde A$ is $\tilde\lambda_0=2$ and that it comes with multiplicity one. Also $\tilde\lambda_N=-2$ has multiplicity one, but all other eigenvalues have (algebraic) multiplicity two.

The connection between the reinforcement matrix $A^*$ and the adjacency matrix $\tilde A$ is the following: A vector $v\in\R^N$ is an eigenvector of $A^*$ if and only if $(v,-v)\in\R^{2N}$ is an eigenvector for $\tilde A$. This connection reflects the fact that particles completing a lap clockwise, in the auto-annihilative process, returns with their signs negated, and can be thought of as an expression of the geometry of a (discrete) M\"obius strip.
The connection is easily verified by decomposing $\tilde A$ into blocks. Let $\tilde B$ be the $N\times N$-matrix with entries
$$
\tilde B_{i, j}=\begin{cases}
1 & \text{ if } \abs{i-j}=1\\
0 & \text{ otherwise }
\end{cases}
$$
and let $\tilde C$ be the $N\times N$-matrix with zero entries except at the SW and NE corners which are 1. We then have
$$
\tilde A=\begin{pmatrix}
\tilde B & \tilde C\\
\tilde C & \tilde B
\end{pmatrix}
$$
and the reinforcement matrix satisfies $A^*=\tilde B-\tilde C$. From here it follows that
$$
(v,-v)\tilde A=(v\tilde B-v\tilde C,v\tilde C-v\tilde B)=(vA^*,-vA^*),
$$
and hence that $vA^*=\lambda v$ if and only if $(v,-v)\tilde A=\lambda (v,-v)$.

Examining the eigenvectors of $\tilde A$ we find that
\begin{equation*}
(\tilde v^{(k)})_{N+j}=\omega^{(N+j)k}=\omega^{Nk}\omega^{jk}=\left\{
\begin{aligned}
-\omega^{jk} &&& \text{if $k$ odd},\\
\omega^{jk} &&& \text{if $k$ even}.
\end{aligned}
\right.
\end{equation*}
That is, for $k$ odd the vector $\tilde v^{(k)}$ is of the form $(v,-v)$, and thus the vectors
$$
u^{(k)}:=\big(1,\omega^{(2k-1)},\omega^{2(2k-1)},\ldots,\omega^{(N-1)(2k-1)}\big)
$$
for $k=1,2,\ldots,N$ are eigenvectors of $A^*$ corresponding to the eigenvalues
$$
\lambda_k:=2\cos(\pi(2k-1)/N).
$$
We observe that the largest positive eigenvalue is $\lambda_1=\lambda_{N}=2\cos(\pi/N)$ and occurs with (algebraic) multiplicity two. If $N$ is odd, then $-2$ is an eigenvalue of multiplicity one, but otherwise all other eigenvalues have (algebraic) multiplicity two.

The eigenvectors $u^{(k)}$ are of course complex-valued. Since $A^*$ is real-valued with real eigenvalues, also the real and imaginary parts of $u^{(k)}$ are eigenvectors associated to $\lambda_k$.\footnote{If $N$ is odd, then $u^{((N+1)/2)}$ is real, but all other vectors have nontrivial real and imaginary parts.}
We note, in particular, that the vectors $v^{(1)}$ and $v^{(2)}$, defined in~\eqref{eq:first_vectors}, are of this form, namely that
\begin{equation}\label{eq:main_vectors}
\begin{aligned}
v^{(1)}=\textup{Re}\,u^{(1)}&=\big(1,\cos(\pi/N),\cos(2\pi/N),\ldots,\cos((N-1)\pi/N)\big)\\
v^{(2)}=\textup{Im}\,u^{(1)}&=\big(0,\sin(\pi/N),\sin(2\pi/N),\ldots,\sin((N-1)\pi/N)\big)
\end{aligned}
\end{equation}
are eigenvectors that spann the eigenspace of the largest positive eigenvalue of $A^*$.

We summarise the information required for subsequent sections.

\begin{prop}\label{prop:largest_eigenv}
The real and imaginary parts of the vectors $u^{(k)}$, for $k\le\lfloor(N+1)/2\rfloor$, form an orthogonal basis in $\R^N$ consisting of vectors of equal length $\sqrt{N/2}$. (If $N$ is odd, then $\textup{Im}\,u^{((N
+1)/2)}$ is the zero-vector and is excluded.)
The largest (positive) eigenvalue of the reinforcement matrix $A^*$ is given by
$$\lambda^\ast=2\cos(\pi/N),$$
and the eigenspace of $\lambda^*$, spanned by the vectors $v^{(1)}$ and $v^{(2)}$, takes the form
$$
\Big\{\big(r\cos\big((j-1)\pi/N+\theta\big)\big)_{j=1}^N:r\in\R,\theta\in[-\pi/2,\pi/2)\Big\}.
$$
\end{prop}

\begin{proof}
First note that the scalar product of the vectors $v^{(1)}$ and $v^{(2)}$ equals
$$
v^{(1)}\cdot v^{(2)}=\sum_{j=1}^{N-1}\cos(j\pi/N)\sin(j\pi/N)=0,
$$
since $\cos(x+\pi/2)\sin(x+\pi/2)$ is an odd function. Hence, $v^{(1)}$ and $v^{(2)}$ are orthogonal. An analogous argument shows that the real and imaginary parts of any vector $u_k$ are orthogonal. Consequently, (with the exception of the eigenvalue $-2$ in the case that $N$ is odd,) the eigenspace of each eigenvalue has dimension two, and so the (geometric) multiplicity of each eigenvalue is two. Moreover, since $A^*$ is symmetric, vectors belonging to different eigenspaces are orthogonal. The real and imaginary parts of the vectors $u_k$ thus form an orthogonal basis that span $\R^N$.

Second, note that $|v^{(1)}|^2+|v^{(2)}|^2=N$ as of the Pythagorean trigonometric identity. Moreover, using the cosine double angle formula, we further find that
$$
|v^{(1)}|^2-|v^{(2)}|^2=\sum_{j=1}^N\big[\cos^2\big((j-1)\pi/N\big)-\sin^2\big((j-1)\pi/N\big)\big]=\sum_{j=1}^N\cos\big(2(j-1)\pi/N\big)=0,
$$
e.g.\ using Lagrange's trigonometric identities. Hence, $v^{(1)}$ and $v^{(2)}$ have both length $\sqrt{N/2}$. An analogous argument shows that the real and imaginary parts of any vector $u^{(k)}$ have length $\sqrt{N/2}$, verifying the first claim of the proposition.

Since the eigenspace of $\lambda^*$ is spanned by the vectors $v^{(1)}$ and $v^{(2)}$, we will need to verify that every linear combination of these vectors is of the form described above. More precisely, we shall show that for every $a,b\in\R$ there exist $r\in\R$ and $\theta\in[-\pi/2,\pi/2)$ such that
\begin{equation}\label{eq:lincomb}
a\cos(x)+b\sin(x)=r\cos(x+\theta).
\end{equation}
In case that $a=0$ this is immediate since $b\sin(x)=b\cos(x-\pi/2)$. For $a\neq0$ the statement follows from the harmonic addition formula with
\begin{equation}\label{eq:harmonic_addition}
r=\sign(a)\sqrt{a^2+b^2}\quad\text{and}\quad\theta=\arctan(-b/a).
\end{equation}
Indeed, the mapping $(a,b)\mapsto(r,\theta)$ forms a bijection from $\R^2$ to $\R\times[-\pi/2,\pi/2)$.
\end{proof}

Since we shall come back to the fact that~\eqref{eq:harmonic_addition} solves~\eqref{eq:lincomb} below, let us, for the sake of completeness, verify that it indeed is true. We first note that $a\cos(x)$ and $b\sin(x)$ are the real parts of the complex numbers $ae^{ix}$ and $-bie^{ix}$, respectively. We further note that for any complex number $z=re^{i\theta}$, the real part of $ze^{ix}$ is the cosine function
$$
\textup{Re}[re^{i\theta}e^{ix}]=r\cos(x+\theta).
$$
Since the sum of $a\cos(x)$ and $b\sin(x)$ is the real part of \emph{some} number of the form $ze^{ix}$, it follows that~\eqref{eq:lincomb} must hold for some $r\in\R$ and $\theta\in[-\pi/2,\pi/2)$.

We move on to express $r$ and $\theta$ in terms of $a$ and $b$. We thus seek to find the modulus and argument of $a-bi$. Its modulus is $\sqrt{a^2+b^2}$. For $a>0$ its argument $\varphi$ satisfies $\tan(\varphi)=-b/a$, so that $\varphi=\arctan(-b/a)$. For $a<0$ we instead get $\varphi=\arctan(-b/a)+\pi$. That is, the argument of $a-bi$ equals $\arctan(-b/a)+\pi\ind_{\{a<0\}}$, so that
$$
a-bi=\sqrt{a^2+b^2}\exp\big(i\big(\arctan(-b/a)+\pi\ind_{\{a<0\}}\big)\big)=\sign(a)\sqrt{a^2+b^2}e^{i\arctan(-b/a)},
$$
and~\eqref{eq:harmonic_addition} follows.
(An alternative derivation uses the cosine angle addition formula.)

\section{Martingale analysis}\label{sec:martingale}

Having examined the eigenstructure of the reinforcement matrix in the preceding section, we proceed with an analysis of the martingales that arise therefrom. Given an eigenvalue $\lambda$ of the reinforcement matrix $A^*$ and any vector $v$ in the eigenspace of $\lambda$, let
$$
M_v(t):=e^{-\lambda t}vZ(t)\quad\text{for }t\ge0.
$$
Here and below, we think of $Z(t)$ as a column vector, so that $vZ(t)$ is equal to the scalar product $v\cdot Z(t)$.
In the case that $\lambda=\lambda^*$, the largest (positive) eigenvalue of $A^*$, and we want to emphasise this fact, we shall then use the notation
$$
M^\ast_v(t):=e^{-\lambda^*t}vZ(t)\quad\text{for }t\ge0.
$$
Throughout we denote by $(\mathcal{F}_t)_{t\ge0}$ the natural filtration where $\mathcal{F}_t=\{Z(s):s\le t\}$.

Below we shall be interested to study the process $(M^\ast_v(t))_{t\ge0}$ for arbitrary vectors $v\in\R^N$, which is defined analogously. However, at the moment we shall focus on eigenvectors, in which case the associated process is a martingale.

\begin{lemma}\label{lemma:mg}
For every eigenvalue $\lambda$ of the reinforcement matrix, and any vector $v$ in the eigenspace of $\lambda$, the stochastic process $(M_v(t))_{t\ge0}$ is a martingale with respect to $(\mathcal{F}_t)_{t\ge0}$.
\end{lemma}

\begin{proof}
We will show that $E[vZ(t)]=e^{\lambda t}vZ(0)$. Due to time homogeneity and the Markov character of the process it will readily follow that
$$
e^{-\lambda(t+s)}E[vZ(t+s)|\mathcal{F}_s]=e^{-\lambda(t+s)}E[vZ(t+s)|Z(s)]=e^{-\lambda s}vZ(s),
$$
so the martingale property holds.

Almost surely, no two clocks ring at the same time. If at time $t$ a clock rings at node $k$, then $vZ(t)$ jumps by $\pm vA^*_k$, where $A^*_k$ is the $k$th column of the matrix $A^*$ and the sign depends on the colour of the ball. Since $v$ is an eigenvector of $A^*$, we have $vA^*=\lambda v_k$, where $v_k$ is the $k$th element of $v$. The number of balls (signs included) at node $k$ is $Z_k(t)$, and each ball rings with intensity $1$, so the expected rate of change equals
$$
\frac{d}{dt}E[vZ(t)]=\sum_{k=1}^NE[vA^*_kZ_k(t)]=\sum_{k=1}^NE[\lambda v_kZ_k(t)]=\lambda E[vZ(t)].
$$
This differential equation has the general solution $E[vZ(t)]=Ce^{\lambda t}$, and the initial condition $E[vZ(0)]=vZ(0)$ gives $E[vZ(t)]=e^{\lambda t}vZ(0)$, as required.
\end{proof}

We next derive a bound on the quadratic variation of the above martingales.

\begin{lemma}\label{lemma:L2bound}
For every eigenvalue $\lambda$ of the reinforcement matrix, every vector $v$ in the eigenspace of $\lambda$ and initial configuration $x\in\Z^N$, we have for all $t\ge s\ge 0$ that
$$
E[M_v(t)^2]-E[M_v(s)^2]\le4|v|^2\|x\|\int_s^t e^{-2(\lambda-1)u}\,du.
$$
\end{lemma}

\begin{proof}
Since each branching event results in at most two new balls, the process $(Z(t))_{t\ge0}$ is dominated by a Galton-Watson (or Bienaym\'e) process, and therefore clearly square integrable. By Lemma~\ref{lemma:mg}, the sequence $(M_v(t))_{t\ge0}$ is a martingale. Its quadratic variation $[M_v](t)$ is defined as the limit, in probability,
$$
[M_v](t):=M_v(0)^2+\lim_{|\mathcal{P}_n|\to0}\sum_{k=1}^n\big(M_v(t_k)-M_v(t_{k-1})\big)^2,
$$
where $\mathcal{P}_n$ is some sequence of partitions of $[0,t]$ with mesh tending to zero. It is well-known that this limit is well-defined and that $(M_v(t)^2-[M_v](t))_{t\ge0}$ is again a martingale which vanishes at $t=0$; see e.g.~\cite[Corollary~2 to Theorem~II.27]{protter90}. In particular, it follows that
\begin{equation}\label{eq:quadid}
E[M_v(t)^2]=E\big[[M_v](t)\big].
\end{equation}

Since almost every realization of $M_v(t)$ has piecewise smooth trajectories, it follows by~\cite[Theorems~II.26 and~II.28]{protter90} that $[M_v](t)$ is a pure jump process, and that
$$
[M_v](t)=\sum_{0\le u\le t}\big(\Delta M_v(u)\big)^2=\sum_{0\le u\le t}\big(M_v(u)-M_v(u-)\big)^2.
$$
That is, if at time $t$ a clock at position $k$ rings, then $[M_v](t)$ jumps by (this time the sign of the ball is irrelevant)
\begin{equation}\label{eq:Delta}
\Delta[M_v](t)=(\Delta M_v(t))^2=e^{-2\lambda t}(vA^*_k)^2=e^{-2\lambda t}\lambda^2v_{k}^2,
\end{equation}
where again $A^*_k$ denotes the $k$th column of $A^*$ and $v_{k}$ the $k$th element of $v$. Consequently, since clocks ring at intensity $1$ and at $k$ there are $|Z_k(t)|$ balls,
\begin{equation}\label{eq:quaddiff}
\frac{d}{dt}E\big[[M_v](t)\big]=\sum_{k=1}^Ne^{-2\lambda t}\lambda^2v_{k}^2E|Z_k(t)|\le4|v|^2e^{-2\lambda t}E\bigg[\sum_{k=1}^N|Z_k(t)|\bigg].
\end{equation}
Since the total number of balls in $Z(t)$ is dominated by a Bienaym\'e-Galton-Watson process with a deterministic offspring distribution resulting in two children, we have
$$
E\bigg[\sum_{k=1}^N|Z_k(t)|\bigg]\le e^{2t}\sum_{k=1}^N|Z_k(0)|=e^{2t}\|x\|.
$$
Integrating~\eqref{eq:quaddiff} over the interval $[s,t]$ leads to the bound
$$
E\big[[M_v](t)\big]-E\big[[M_v](s)\big]\le4|v|^2\|x\|\int_s^t e^{-2(\lambda-1)u}\,du.
$$
The lemma now follows from~\eqref{eq:quadid}.
\end{proof}

From the bound on the quadratic variation we may conclude boundedness, and thus almost sure and $L^2$ convergence, of the martingales associated to the largest eigenvalue.

\begin{lemma}\label{lemma:L2mg}
Suppose that $N\ge4$. Then, for every nonzero vector $v$ in the eigenspace of $\lambda^*$ and every nonzero initial configuration $x\in\Z^N\setminus\{0\}$, the process $(M^*_v(t))_{t\ge0}$ converges almost surely and in $L^2$ to a limiting random variable $W=W(v,x)$ with finite mean and variance. Moreover, $P\big(W(v,x)\neq0\big)>0$.
\end{lemma}

\begin{proof}
By Lemma~\ref{lemma:mg}, the process $(M^*_v(t))_{t\ge0}$ is a martingale. For $N\ge4$ we have that $\lambda^*=2\cos(\pi/N)>1$. Consequently, by Lemma~\ref{lemma:L2bound},
\begin{equation}\label{eq:Mtilde}
E[M^*_v(t)^2]\le E[M^*_v(0)^2]+4|v|^2\|x\|\int_0^te^{-2(\lambda^*-1)u}\,du
\end{equation}
is bounded, and $(M^*_v(t))_{t\ge0}$ thus convergent almost surely and in $L^2$.

We show next that $W\neq0$ with positive probability. Note first that the quadratic variation of $M^*_v(t)$, by definition, is non-decreasing as a function of $t$. Consequently, by~\eqref{eq:quadid}, $E[M^*_v(t)^2]$ is non-decreasing and approaches $E[W^2]$ as $t\to\infty$. Hence, it will suffice to show that $E[M^*_v(1)^2]>0$.

Either $M^*_v(0)^2>0$, in which case $E[M^*_v(1)^2]>0$ holds trivially. Or, as long as there is at least one particle present at time zero, there is positive probability of a (series of nucleations leading to a) nucleation at some node $k$ at which $v$ is nonzero, before $t=1$. By~\eqref{eq:Delta}, this has a positive contribution to the quadratic variation, which by~\eqref{eq:quadid} gives $E[M^*_v(1)]>0$. Consequently, $P(W\neq0)>0$.
\end{proof}

We remark that, for any $x\in\Z^N$, $a,b\in\R$ and vectors $u,v$ in the eigenspace of $\lambda^*$, the following linearity property of the limiting variable is immediate from definition:
\begin{equation}\label{eq:linearity}
W(au+bv,x)=aW(u,x)+bW(v,x).
\end{equation}
As a consequence, since the eigenspace of $\lambda^*$ is spanned by the vectors $v^{(1)}$ and $v^{(2)}$, the limiting function $W(\,\cdot\,,x)$ is determined by the values $W(v^{(1)},x)$ and $W(v^{(2)},x)$.

Above we have supposed that $v$ is in the eigenspace of $\lambda^*$ when considering the process $(M^*_v(t))_{t\ge0}$. We shall henceforth relax that assumption. Of course, the process remains well-defined for arbitrary vectors $v\in\R^N$, although it is no longer a martingale.

\begin{lemma}\label{lemma:mg_vanish}
Suppose that $N\ge4$. For every eigenvalue $\lambda\neq\lambda^*$ of the reinforcement matrix, any vector $v$ in the eigenspace of $\lambda$, and any initial configuration $x\in\Z^N$, the process $(M^*_v(t))_{t\ge0}$ vanishes almost surely as $t\to\infty$.
\end{lemma}

\begin{proof}
We shall first prove that for every $\delta>0$ we have, almost surely,
\begin{equation}\label{eq:subsequence}
M^*_v(\delta n)\to 0\quad\text{as }n\to\infty.
\end{equation}
Indeed, since $M^*_v(t)=e^{-(\lambda^*-\lambda)}M_v(t)$, and $v$ is a vector in the eigenspace of $\lambda$, an application of Lemma~\ref{lemma:L2bound} gives that
\begin{equation}\label{eq:expbound}
E[M^*_v(t)^2]
\le e^{-2(\lambda^*-\lambda)t}E[M_v(0)^2]+\left\{
\begin{aligned}
&4|v|^2\|x\|te^{-2(\lambda^*-\lambda)t} &&\text{if }\lambda\ge1,\\
&4|v|^2\|x\|\frac{1}{1-\lambda}e^{-2(\lambda^*-1)t} &&\text{if }\lambda<1.
\end{aligned}
\right.
\end{equation}
Hence, for $\lambda<\lambda^*$, $E[M^*_v(t)^2]$ decays exponentially fast in $t$. It follows from~\eqref{eq:expbound} that
$$
E\bigg[\sum_{n\ge1}M^*_v(\delta n)^2\bigg]=\sum_{n\ge1}E[M^*_v(\delta n)^2]<\infty,
$$
so, almost surely, the series $\sum_{n\ge1}M^*_v(\delta n)^2$ converges, and~\eqref{eq:subsequence} follows.

We next show that for only finitely many $n$ the sequence $(M^*_v(t))_{t\ge0}$ may deviate far during the interval $[\delta n,\delta(n+1)]$. Let
\begin{align*}
A_n&:=\big\{M^*_v(\delta n)\in(-\delta,\delta)\big\},\\
B_n&:=\big\{
|M_v(t)-M_v(\delta n)|<\delta e^{(\lambda^*-\lambda)\delta n}\text{ for all }t\in[\delta n,\delta(n+1)]\big\}.
\end{align*}
By~\eqref{eq:subsequence} the events $A_n$ occur for all but finitely many $n$ almost surely. Using the Doob-Kolmogorov maximal inequality and Lemma~\ref{lemma:L2bound} we obtain
\begin{align*}
P(B_n^c)&\le\frac{1}{\delta^2}e^{-2(\lambda^*-\lambda)\delta n}E\big[(M_v(\delta(n+1))-M_v(\delta n))^2\big]\\
&\le\frac{1}{\delta^2}4|v|^2\|x\|e^{-2(\lambda^*-\lambda)\delta n}\int_{\delta n}^{\delta(n+1)}e^{-2(\lambda-1)u}\,du\\
&\le\frac{1}{\delta^2}4|v|^2\|x\|e^{-2(\lambda^*-1)\delta n+2\delta}.
\end{align*}
For $N\ge4$ we have $\lambda^*>1$, so the probabilities $P(B_n^c)$ are summable. Borel-Cantelli gives that $B_n$ will occur for all but finitely many $n$. Finally, on the event $A_n\cap B_n$ we have for $t\in[\delta n,\delta(n+1)]$ that
\begin{align*}
\big|M^*_v(t)-M^*_v(\delta n)\big|&\le e^{-(\lambda^*-\lambda)t}\big|M_v(t)-M_v(\delta n)\big|+\big(e^{-(\lambda^*-\lambda)\delta n}-e^{-(\lambda^*-\lambda)t}\big)|M_v(\delta n)|\\
&\le\delta e^{-(\lambda^*-\lambda)(t-\delta n)}+\delta\big(1-e^{-(\lambda^*-\lambda)(t-\delta n)}\big)=\delta,
\end{align*}
and so that $M^*_v(t)\le 2\delta$ for $t\in[\delta n,\delta(n+1)]$. Since $\delta>0$ was arbitrary, we conclude that $M^*_v(t)\to0$ almost surely as $n\to\infty$.
\end{proof}

As a consequence of the above analysis, we may conclude that $(Z(t))_{t\ge0}$, once properly rescaled, converges to a nontrivial limit as a vector in $\R^N$. Moreover, the limit is a linear combination of the vectors $v^{(1)}$ and $v^{(2)}$ from~\eqref{eq:first_vectors} and~\eqref{eq:main_vectors}, and is thus an element of the eigenspace of the largest eigenvector $\lambda^*$.

\begin{prop}\label{prop:limit}
For any initial configuration $x\in\Z^N$ we have, almost surely, that
$$
\lim_{t\to\infty}e^{-\lambda^* t}Z(t)=\frac{W(v^{(1)},x)}{|v^{(1)}|^2}v^{(1)}+\frac{W(v^{(2)},x)}{|v^{(2)}|^2}v^{(2)}.
$$
\end{prop}

\begin{proof}
Recall, from Proposition~\ref{prop:largest_eigenv}, that the real and imaginary parts of the complex-valued eigenvectors $u^{(k)}$ form an orthogonal basis in $\R^N$. We shall below denote these base vectors by $v^{(1)},v^{(2)},\ldots,v^{(N)}$, where $v^{(1)}$ and $v^{(2)}$ are the real and imaginary parts of $u^{(1)}$, and the remaining vectors are ordered arbitrarily. Every vector $v\in\R^N$ can thus be uniquely expressed as $v=a_1v^{(1)}+\ldots+a_Nv^{(N)}$ for some $a_1,\ldots,a_N$ in $\mathbb{R}$. By Lemmas~\ref{lemma:L2mg} and~\ref{lemma:mg_vanish} we have, almost surely, that
\begin{equation}\label{eq:gen_lim}
\lim_{t\to\infty}M^*_v(t)=\lim_{t\to\infty}\big[a_1M^*_{v^{(1)}}(t)+\ldots+a_NM^*_{v^{(N)}}(t)\big]=a_1W(v^{(1)},x)+a_2W(v^{(2)},x).
\end{equation}
In particular, this holds for the coordinate vectors $e_1,e_2,\ldots,e_N$ of $\mathbb{R}^N$. Since $(v^{(k)})_{k=1}^N$ forms an orthogonal basis, we have for $j=1,2,\ldots,N$ that
$$
e_j=\frac{(v^{(1)})_j}{|v^{(1)}|^2}v^{(1)}+\ldots+\frac{(v^{(N)})_j}{|v^{(N)}|^2}v^{(N)},
$$
where $(v^{(k)})_j$ is the $j$th coordinate of $v^{(k)}$. By~\eqref{eq:gen_lim}, we obtain, almost surely, that
$$
\lim_{t\to\infty}M^*_{e_j}(t)=\frac{(v^{(1)})_j}{|v^{(1)}|^2}W(v^{(1)},x)+\frac{(v^{(2)})_j}{|v^{(2)}|^2}W(v^{(2)},x),
$$
as required.
\end{proof}

\section{The martingale limit}\label{sec:limit}

As a consequence of the martingale analysis of the previous section, $(e^{-\lambda^*t}Z(t))_{t\ge0}$ converges almost surely to an element in the eigenspace of $\lambda^*$. The goal of this section is to describe the limiting element more closely, in order to complete the proof of Theorem~\ref{thm:signed}.
A key step will be a coupling construction, that will allow us to compare the evolution of the signed competition process with the evolution of the same in the case that the first nucleation had never taken place. Examining the difference between the two will allow us to draw conclusions regarding the limiting variables $W=W(v,x)$. A similar coupling construction was previously introduced in~\cite{ahlgrijanmor19}, and put to use for a similar purpose.

\subsection{Coupling construction}\label{sec:coupling}

Our first goal will be to construct a larger system of particles, from which we can read out the evolution of different versions of the auto-annihilative process. The larger system will consist of red, blue and purple particles, each of which can be either positive or negative. Each colour, when considered on its own, will evolve according to the dynamics of the signed competition process $(Z(t))_{t\ge0}$. In particular, balls of the same colour and opposing signs (when present at the same site) annihilate on a one-for-one basis. Balls of different colour (when present at the same site) interact according to the following rules.
\begin{itemize}
\item Red and blue of the same sign do not interact; red and blue of opposing signs result in a purple particle with the same sign as the red particle.
\item Red and purple of the same sign do not interact; red and purple of opposing signs result in a blue particle with the same sign as the red particle.
\item Blue and purple of opposing signs do not interact; blue and purple of the same sign result in a red particle of the same sign.
\end{itemize}

The rationale behind the rules of interaction between particles of different colour is that a purple ball should be thought of as temporary bond between a red and a blue particle of opposing signs. The bond is broken when the pair interacts with another red or blue particle whose sign opposes the sign of the particle of the same colour present in the bond. When the bond is broken the same colour particles (of opposing sign) annihilate, and the particle remaining is set free.

Every particle in the system is given a Poisson clock when born. For definiteness we may assume that the purple particle that arises from the merger of a red and a blue particle adopts the clock of its red constituent, and when the bond is eventually broken, the particle set free returns to follow its original clock. It is straightforward to verify that the above larger system of particles is again dominated by a Bienaym\'e-Galton-Watson process, and is thus a pure-jump Markov process without explosions.

We next describe the initial configuration of our larger system of coloured particles. Let $x\in\Z^N\setminus\{0\}$ be any nonzero configuration, and let $X$ and $T$ be auxiliary random variables, independent of each other and everything else, such that $X$ takes values in $\{1,2,\ldots,N\}$ with
$$
P(X=k):=\frac{|x_k|}{\|x\|}\quad\text{for }k=1,2,\ldots,N,
$$
and $T$ is exponentially distributed with intensity $\|x\|$.
Add signed red balls to the nodes of $C_N$ according to $x$, and sample $X$. On the event that $\{X=k\}$, send one blue ball along each edge adjacent to node $k$ with sign corresponding to the sign of $x_k$. (If one of the blue balls traverses the signed edge, then its sign is changed as usual.) The initial configuration of blue balls, on the event $\{X=k\}$ thus equals $A^*_k\cdot\sign(x_k)$, where $A^*_k$ is the $k$-th column of the reinforcement matrix $A^*$. The blue balls instantaneously interact with any red balls of opposing sign, if any. The resulting configuration will be taken as the initial configuration of the larger system of particles.

Our aim is to compare the evolution of the auto-annihilative process $(Z(t))_{t\ge0}$, started from some configuration $x\in\Z^N\setminus\{0\}$, with the evolution of itself in the case that the first nucleation had never taken place. The former corresponds to the evolution of the red particles initially placed on the vertices of $C_N$. Note that the random configuration of blue particles initially added to the configuration of red particles is equivalent to the result of the first nucleation of the auto-annihilative process started from $x$. Since we want to track the evolution of the process both in the case that this nucleation does and does not take place, we want to be able to distinguish between these particles being present and not. This is the purpose of the purple particles, which should thus be thought of as present (red) particles in absence of the first nucleation, and as absent (annihilated) particles in the presence of the first nucleation.

To make the comparison between the two versions of the signed competition process precise, let $R(t)$, $B(t)$ and $P(t)$ be vectors encoding the number of red, blue and purple particles present at the nodes of $C_N$ in the system at time $t\ge0$. For $t\ge0$, let
$$
Z'(t):=R(t)+P(t),
$$
and
\begin{equation*}
Z''(t):=\left\{
\begin{aligned}
& x & \text{if }t<T,\\
& R(t-T)+B(t-T) & \text{if }t\ge T.
\end{aligned}
\right.
\end{equation*}

Note that $T$ is equal in distribution to the time of the first nucleation in the signed competition process with initial configuration $x$, and the random configuration of blue particles initially generated is equal in distribution to the particles generated in the first nucleation. Consequently, the vector $Z'(t)$ can be interpreted as the configuration at time $t$ in case the first nucleation is ignored and time reset at this point; delaying the start of $(Z'(t))_{t\ge0}$ for $T$ units of time results in a process where the first nucleation is suppressed. The vector $Z''(t)$ can similarly be interpreted as the configuration at time $t$ if the first nucleation is allowed to take place (at time $T$).

In addition, for $t\ge0$, we set 
$$
D(t):=B(t)-P(t).
$$
Note that $D(t)$ denotes the difference between the two processes $Z'$ and $Z''$ in the sense that
\begin{equation}\label{eq:couplingdiff}
Z''(t)=x\ind_{\{t<T\}}+\big[Z'(t-T)+D(t-T)\big]\ind_{\{t\ge T\}}.
\end{equation}

Our next lemma relates the larger system of particles to the original process.

\begin{lemma}\label{lemma:delayedcoupling}
Let $x\in\Z^N\setminus\{0\}$ be any nonzero configuration, and consider the above system of coloured competition. Then, the processes $(Z'(t))_{t\ge0}$ and $(Z''(t))_{t\ge0}$ are equal in distribution to the signed competition process $(Z(t))_{t\ge0}$ with initial configuration $x$. Conditional on the event $\{X=k\}$, the process $(D(t))_{t\ge0}$ is equal in distribution to the signed competition process with initial configuration $A^*_k\cdot\sign(x_k)$.
\end{lemma}

\begin{proof}
Note that $Z'(0)=Z''(0)=x$ by definition. Since purple particles follow the clocks of their red components, $(Z'(t))_{t\ge0}$ is readily seen to be equal to a version of the signed competition process with initial configuration $x$. 

Let $(Z(t))_{t\ge0}$ be a version of the signed competition process with initial configuration $x$. The time of the first nucleation is then distributed as $T$, the position where the first nucleation takes place is distributed as $X$, and on the event that the first nucleation takes place at position $k$ the result of the nucleation is addition by the vector $A^*_k\cdot\sign(x_k)$. This shows that $(Z(t))_{t\ge0}$ and $(Z''(t))_{t\ge0}$ are equal in distribution up to and including the time of the first nucleation. Since $(Z''(t))_{t\ge0}$ does not differentiate between red and blue particles, but ignore those that are purple, it follows from the strong Markov property that $(Z(t))_{t\ge0}$ and $(Z''(t))_{t\ge0}$ remain equal in distribution also after the time of the first nucleation. (One may imagine that all clocks of the system are reset at the time of every nucleation, as this does not affect the distribution of the system.)

Finally, since purple particles adopt the sign of their red constituent, we have on the event $\{X=k\}$ that $D(0)=A^*_k\cdot\sign(x_k)$. Moreover, thinking of each purple particle as a blue particle temporarily paired with a red particle of opposing sign, the number of particles that are either blue or (negative) purple correspond to the number of blue particles that would be present in absence of red particles. In addition, however, a paired blue particle temporarily follows the clock of its red partner. That this (possibly temporary) change of clocks does not affect the law of the process follows from elementary properties of the Poisson process, together with the strong Markov property.
\end{proof}

\subsection{Limiting behaviour}

Let $v$ be any nonzero vector in the eigenspace of the largest eigenvalue $\lambda^*$. The distribution of the limiting variable $W(v,x)=\lim_{t\to\infty}M^*_v(t)$ will depend both on $v$ and the initial configuration $x$. Next we make use of the above coupling construction to show that $W=W(v,x)$ is almost surely nonzero.

\begin{lemma}\label{lemma:limprob}
Suppose that $N\ge4$ and let $v$ be any nonzero vector in the eigenspace of $\lambda^*$. There exists $c>0$ such that for every $x\in\Z^N\setminus\{0\}$ we have
$$
P\big(W(v,x)=0\big)<1-c.
$$
\end{lemma}

\begin{proof}
Consider the coupling construction from Section~\ref{sec:coupling}. By Lemma~\ref{lemma:delayedcoupling}, the sequences $(Z'(t))_{t\ge0}$, $(Z''(t))_{t\ge0}$ and $(D(t))_{t\ge0}$ are all versions of the auto-annihilative growth process. Let
$$
M'_v(t):=e^{-\lambda^* t}vZ'(t),\, M''_v(t):=e^{-\lambda^* t}vZ''(t)\text{ and }M^D_v(t):=e^{-\lambda^* t}vD(t).
$$
By Lemma~\ref{lemma:L2mg}, the sequences $(M'_v(t))_{t\ge0}$, $(M''_v(t))_{t\ge0}$ and $(M^D_v(t))_{t\ge0}$ converge almost surely and in $L^2$ to random variables $W'=W'(v,x)$, $W''=W''(v,x)$ and $W^D=W^D(v,x)$, respectively. From equation~\eqref{eq:couplingdiff} we deduce that the limiting variables satisfy the almost sure relation
\begin{equation}\label{eq:limitdiff}
W''=e^{-\lambda^* T}\big(W'+W^D\big).
\end{equation}

Both of the limiting variables $W'$ and $W''$ are, again by Lemma~\ref{lemma:delayedcoupling}, equal in distribution to $W(v,x)$, and, conditioned on the event $\{X=k\}$, the variable $W^D$ is equal in distribution to $W(v,y_k)$ with $y_k=A^*_k\cdot\sign(x_k)$. For each nonzero initial configuration $x$, it follows from Lemma~\ref{lemma:L2mg} that $W(v,x)$ has positive probability of being nonzero. Since $X$ takes on a finite number of values, it follows that there exists a constant $c>0$, not depending on $x$, such that for each $k=1,2,\ldots,N$ we have
$$
P\big(W^D\neq0\big|X=k\big)\ge 2c.
$$
Via the law of total probability we obtain, uniformly over $x\in\Z^N\setminus\{0\}$, that
\begin{equation}\label{eq:WDbound}
P\big(W^D\neq0\big)\ge 2c.
\end{equation}

Note that on the event that $W^D\neq0$, by~\eqref{eq:limitdiff}, we cannot (with positive probability) have both $W'=0$ and $W''=0$. Hence, by~\eqref{eq:WDbound}, we have for all $x\in\Z^N\setminus\{0\}$ that
$$
P\big(W'=0,W''=0\big)\le P\big(W^D=0\big)\le1-2c.
$$
Consequently,
$$
P(W'=0)=P\big(W'=0,W''\neq0\big)+P\big(W'=0,W''=0\big)\le P(W''\neq0)+1-2c,
$$
which yields
$$
P(W'=0)+P(W''=0)\le 2-2c.
$$
Since both the limiting variables $W'$ and $W''$ are equal in distribution to $W(v,x)$, we conclude that $P(W(v,x)=0)\le1-c$ uniformly in $x$, as required.
\end{proof}

Using the Markov character of the process, we deduce next that $W(v,x)$ is almost surely nonzero.

\begin{lemma}\label{lemma:Wnot0wp1}
Suppose that $N\ge4$ and let $v$ be any nonzero vector in the eigenspace of $\lambda^*$. Then, for every $x\in\Z^N\setminus\{0\}$, we have
$$
P\big(W(v,x)=0\big)=0.
$$
\end{lemma}

\begin{proof}
Recall that $(\mathcal{F}_t)_{t\ge0}$ denotes the natural filtration where $\mathcal{F}_t$ is the sigma algebra generated by $\{Z(s):s\le t\}$. By the L\'evy 0-1-law we have almost surely that
$$
\lim_{t\to\infty}P\big({W(v,x)}=0\big|\mathcal{F}_t\big)=\ind_{\{W(v,x)=0\}}.
$$
Moreover, due to the Markov character of the process, Lemma~\ref{lemma:limprob} gives that with probability one
$$
P\big(W(v,x)=0\big|\mathcal{F}_t\big)=P\big(W(v,Z(t))\big|Z(t)\big)\le1-c.
$$
It follows that $\ind_{\{W(v,x)=0\}}=0$, and thus that $W(v,x)\neq0$, almost surely.
\end{proof}

%
%

%

In addition, we complement the preceding lemma with the following concentration bound.

\begin{lemma}\label{lemma:concentration}
Suppose that $N\ge4$. For every initial configuration $x\in\Z^N$ and nonzero vector $v$ in the eigenspace of $\lambda^*$, we have
$$
P\big(\big|W(v,x)-v\cdot x\big|\ge\alpha\big)\le\frac{5|v|^2\|x\|}{\alpha^2}.
$$
\end{lemma}

\begin{proof}
By Lemma~\ref{lemma:L2mg} the process $(M^*_v(t))_{t\ge0}$ converges almost surely and in $L^2$. Consequently, the limit $W=W(v,x)$ has finite mean and variance, and satisfies
$$
E[W]=\lim_{t\to\infty}E[M^*_v(t)]=E[M^*_v(0)]=v\cdot x
$$
and, together with Lemma~\ref{lemma:L2bound},
$$
E[W^2]=\lim_{t\to\infty}E[M^*_v(t)^2]\le E[M^*_v(0)^2]+4|v|^2\|x\|\int_0^\infty e^{-2(\lambda^*-1)u}\,du
$$
Since $E[M^*_v(0)^2]=(v\cdot x)^2=E[W]^2$, computing the integral and using Chebyshev's inequality gives that
$$
P\big(|W-v\cdot x|\ge\alpha\big)\le\frac{Var(W)}{\alpha^2}\le\frac{5|v|^2\|x\|}{\alpha^2},
$$
as required.
\end{proof}

\subsection{Proof of Theorem~\ref{thm:signed}}

Suppose that $N\ge4$ and that $x\in\Z^N\setminus\{0\}$. By Proposition~\ref{prop:limit} we have, almost surely,
\begin{equation}\label{eq:final_limit}
\lim_{t\to\infty}e^{-\lambda^*t}Z(t)=\frac{W(v^{(1)},x)}{|v^{(1)}|^2}v^{(1)}+\frac{W(v^{(2)},x)}{|v^{(2)}|^2}v^{(2)}.
\end{equation}
By Lemma~\ref{lemma:Wnot0wp1} we have $W(v^{(1)},x)\neq0$ and $W(v^{(2)},x)\neq0$ with probability one. Consequently, by the harmonic addition formula, the limit in~\eqref{eq:final_limit} can be written on the form $(R\cos((j-1)\pi/N+S))_{j=1,\ldots,N}$ for random variables $R=R(x)$ and $S=S(x)$ satisfying
$$
R=\sign\big(W(v^{(1)},x)\big)\frac{\sqrt{W(v^{(1)},x)^2+W(v^{(2)},x)^2}}{|v^{(1)}|^2}\quad\text{and}\quad S=\arctan\bigg(-\frac{W(v^{(2)},x)}{W(v^{(1)},x)}\bigg).
$$
(Recall that $|v^{(1)}|=|v^{(2)}|=\sqrt{N/2}$; see Proposition~\ref{prop:largest_eigenv}.)

\begin{claim}\label{claim:continuity}
For every $x\in\Z^N\setminus\{0\}$ (the distribution functions of) the random variables $R(x)$ and $S(x)$ are continuous.
\end{claim}

\begin{proof}[Proof of claim]
By~\eqref{eq:linearity}, for every $\theta\in(-\pi/1,\pi/2)$ the event $\{S(x)=\theta\}$ is equivalent to
$$
\tan(\theta)W(v^{(1)},x)+W(v^{(2)},x)=W\big(\tan(\theta)v^{(1)}+v^{(2)},x\big)
$$
being zero, which by Lemma~\ref{lemma:Wnot0wp1} occurs with probability zero. Hence, $S$ is continuous.

Next, consider the coupling construction from Section~\ref{sec:coupling}. Using~\eqref{eq:limitdiff} we obtain that $R(x)$ is equal in distribution to
$$
e^{-\lambda^*T}\sqrt{\big(W'(v^{(1)},x)+W^D(v^{(1)},x)\big)^2+\big(W'(v^{(2)},x)+W^D(v^{(2)},x)\big)^2}.
$$
Since $T$ is continuous, and independent from $W'$ and $W^D$, which in turn are almost surely non-zero, also $R$ is continuous.
\end{proof}

To complete the proof of the theorem it remains to prove that the random vector $(R,S)$ is fully supported on $\R\times[-\pi/2,\pi/2)$.

\begin{claim}\label{claim:support}
For every $x\in\Z^N\setminus\{0\}$, $(r,\theta)\in\R\times(-\pi/2,\pi/2)$ and $\eps>0$ we have
$$
P\Big(R(x)\in(r-\eps,r+\eps),S(x)\in(\theta-\eps,\theta+\eps)\Big)>0.
$$
\end{claim}

\begin{proof}[Proof of claim]
Let $(r,\theta)\in\R\times(-\pi/2,\pi/2)$ and $\eps>0$ be arbitrary. By symmetry, we may assume that $r>0$. Let $c=-\tan(\theta)$ and set $y=v^{(1)}/|v^{(1)}|^2+cv^{(2)}/|v^{(2)}|^2$. Then $v^{(1)}\cdot y=1$ and $v^{(2)}\cdot y=c$, since $v^{(1)}$ and $v^{(2)}$ are orthogonal. For $M>0$ define
$$
B_M(y):=\left\{z\in\Z^N:\|z-My\| <\sqrt{\frac{2M}{N}}\right\}.
$$
Again, by Proposition \ref{prop:largest_eigenv}, $|v^{(1)}|=|v^{(2)}|=\sqrt{N/2}$. It follows that, for $z\in B_M(y)$, we have $v^{(1)}\cdot z\in(M-\sqrt{M},M+\sqrt{M})$ and $v^{(2)}\cdot z\in(Mc-\sqrt{M},Mc+\sqrt{M})$. In addition, by Lemma~\ref{lemma:concentration}, there exists $\alpha=\alpha(c)$ such that for all $M\ge1$, $z\in B_M(y)$ and $i=1,2$ we have
$$
P\big(\big|W(v^{(i)},z)-v^{(i)}\cdot z\big|\ge\alpha\sqrt{M}\big)\le\frac{5|v^{(i)}|^2\|z\|}{\alpha^2M}\le1/4.
$$
Consequently, for every $M\ge1$ and $z\in B_M(y)$, we have with probability at least $1/2$ that
\begin{equation}\label{eq:conc_bound}
\big|W(v^{(1)},z)-M\big|<(1+\alpha)\sqrt{M}\quad\text{and}\quad\big|W(v^{(2)},z)-Mc\big|<(1+\alpha)\sqrt{M}.
\end{equation}

Note that~\eqref{eq:conc_bound} implies that $W(v^{(1)},z)>0$ and, for some constant $C$, that
\begin{equation}\label{eq:quotient}
\bigg|\frac{W(v^{(2)},z)}{W(v^{(1)},z)}-c\bigg|<\frac{C}{\sqrt{M}}\quad\text{and}\quad\Big|\sqrt{W(v^{(1)},z)^2+W(v^{(2)},z)^2}-\sqrt{1+|c|^2}M\Big|<CM^{3/4}.
\end{equation}
Due to the Markov character of $(Z(t))_{t\ge0}$, conditional on the event $\{Z(t)=z\}$, the limiting variable $W(v,x)$ is equal in distribution to the limiting variable of a signed competition process with initial condition $z$ and delayed for time $t$, i.e.\
\begin{equation}\label{eq:joaozinho}
W(v,x)\stackrel{d}{=}e^{-\lambda^*t}W(v,z).
\end{equation}
Next, let $B:=(r-\eps,r+\eps)\times(\theta-\eps,\theta+\eps)$. Suppose that $M>r$ and pick $t>0$ so that $r=e^{-\lambda^*t}\sqrt{1+|c|^2}M/|v^{(1)}|^2$. Combining~\eqref{eq:quotient} and~\eqref{eq:joaozinho}, and using the fact that $\arctan$ is a continuous and strictly increasing function, we obtain that for any nonzero initial configuration $x$ and $\eps>0$ that
$$
P\big((R(x),S(x))\in B\big|Z(t)=z\big)=P\big((R(z),S(z))\in B\big)\ge1/2,
$$
which gives
$$
P\big((R(x),S(x))\in B\big)\ge\frac12P\big(Z(t)\in B_M(y)\big).
$$

It remains to show that for any nonzero initial configuration $x$, $t>0$ and all large values of $M$ we have
$$
P\big(Z(t)\in B_M(y)\big)>0.
$$
Note that it will suffice to show that there exists a finite number of nucleations that result in a configuration in $B_M(y)$ for large $M$. Since $y$ is a linear combination of $v^{(1)}$ and $v^{(2)}$ we have that $y=\rho\cos((j-1)\pi/N+\theta)$ for some $\rho\in\R\setminus\{0\}$ and $\theta\in(-\pi/2,\pi/2)$. By symmetry we may assume that $\rho>0$. We may further assume that $x_1\ge1$, since from any configuration $x$ we may reach a configuration which has positive particles at node 1.

We obtain a configuration $z$ from $x$ by, for $j=1,2,\ldots,N$, in that order, letting one of the particles at $j$ nucleate $\lfloor M\rho\cos((j-1)\pi/N+\theta)\rfloor$ times. This will result in
$$
\lfloor M\rho\cos((j-2)\pi/N+\theta)\rfloor+\lfloor M\rho\cos(j\pi/N+\theta)\rfloor=2M\rho\cos((j-1)\pi/N+\theta)\cos(\pi/N)\pm2
$$
particles added at position $j$. For large $M$ this vastly exceeds the number of particles initially present at position $j$, resulting in the configuration $z$ obtained from $x$ satisfying the equation $\|z-My\|<\sqrt{2M/N}$, as required.
\end{proof}

Claims~\ref{claim:continuity} and~\ref{claim:support} together complete the proof of Theorem~\ref{thm:signed}.

\section{Non-coexistence for the competing urn process}\label{sec:proof}

We are now in position to prove Theorem~\ref{thm:urn}, which we shall deduce from Theorems~\ref{thm:AGJM} and~\ref{thm:signed}. With this, we rule out coexistence between any number of competing types in the competing urn scheme on a cycle.

Fix $N\ge K\ge 2$, and consider the competing urn scheme with $K$ types on a cycle of length $N$. Consider an arbitrary (nonzero) configuration of balls of $K$ types positioned into the $N$ urns so that no urn contains balls of more than one type. We may represent such a configuration by an element $y\in\{0,1,\ldots\}^{\{1,2,\ldots,N\}\times\{1,2,\ldots,K\}}$, where $y_{i,j}$ denotes the number of balls at vertex $i$ of type $j$.

Recall, from Section~\ref{sec:AGJM}, the definition of tribes and fronts of a configuration of balls on $C_N$ of $K$ different types. Let $\tau(y)$ denote the number of tribes in the configuration $y$. If $\tau(y)\ge2$, then the number of fronts in $y$ is also equal to $\tau(y)$. (If there is only one tribe, then there are no fronts.) To prove Theorem~\ref{thm:urn} we will need to show that the competing urn process, almost surely, eventually consists of a single remaining tribe. Since, as the competing urn process evolves, the number of tribes can only go down, it will suffice to show that, almost surely, the number of tribes eventually goes down. We shall do this in two steps, depending on whether the number of tribes is even or odd.

Suppose first that the competing urn process is run from an initial configuration $y$ consisting of an even number of tribes. We associate to the configuration $y$ a configuration $x\in\Z^N$ by encoding the balls in each tribe by $+1$s and $-1$s in an alternating fashion. More precisely, we first label the tribes from 1 to $\tau(y)$ in a cyclic clockwise manner, starting with the tribe to which the vertex labeled 1 belongs. (If this vertex is vacant and belongs to two tribes, start with the clockwise-most one.) Second, for $i=1,2,\ldots,N$, if $\ell_i$ denotes the label of the tribe to which node $i$ belongs, we then set
\begin{equation}\label{eq:y_config}
x_i:=(-1)^{\ell_i-1}\sum_{j=1}^Ky_{i,j}.
\end{equation}
Notice that $|x_i|$ denotes the number of balls at position $i$ in the configuration $y$.
The definitions of tribes and fronts extend straightforwardly to configurations $x\in\Z^N$, by regarding them as two-type configurations on $C_N$. Moreover, since  $\tau(y)$ is even, then particles in the first tribe are labeled positively and particles of the last tribe negatively, which implies that the number of tribes in $x$ equals the number of tribes in $y$.

\begin{claim}\label{claim:even}
Suppose that $\tau(y)$ is even. Then there exists a coupling between two competing urn process started in $y$ and in $x$, respectively, such that, with probability one, at each time $t\ge0$ either both processes have $\tau(y)$ tribes remaining or neither does.
\end{claim}

\begin{proof}[Proof of claim]
Couple the two processes so that they evolve according to the same clocks. Then, as long as no tribe becomes extinct, each nucleation has the same effect in both processes.
\end{proof}

By Theorem~\ref{thm:AGJM} it follows that, with probability one, for the process starting in $x$ all vertices will eventually be occupied by either positive or negative particles. That is, the process will eventually consist of a single remaining tribe. In particular, if the number of tribes started out even, then it will eventually go down by one. Together with Claim~\ref{claim:even} we conclude that, with probability one, the $K$-type competing urn process starting from a configuration $y$ with an even number of tribes will eventually consist of at most $\tau(y)-1$ tribes.

Once the number of tribes decreases, the coupling in Claim~\ref{claim:even} may cease to hold. This is due to the fact that the encoding from $y$ to $x$ treats two tribes at an even distance (in terms of labels) as being of equal type, whereas in the original configuration $y$ may consist of different types. Indeed, in the process started from $x$ the number of fronts will decrease by 2 each time a tribe is eliminated, whereas in the process started from $y$ the number of fronts will decrease by either 1 or 2, depending on whether the tribes clockwise and counterclockwise of the eliminated tribe are of different types or not. Hence, the above argument only shows that in the competing urn process starting from a configuration with an even number of tribes, at least one tribe will be eliminated, almost surely.

We shall need a separate argument when the number of tribes is odd. So, suppose next that the number of tribes $\tau(y)$ is odd and greater than or equal to 3. (When there is only one tribe there is nothing to prove.) We again label the tribes of $y$ from 1 to $\tau(y)$ in a cyclic clockwise fashion. By rotating the labelling of the vertices (if necessary) we may assume that the tribe labeled 1 has its counterclockwise-most endpoint at the vertex labeled 1. Again we obtain from $y$ a configuration $x\in\Z^N$ as in~\eqref{eq:y_config}. Since the number of tribes is odd, the first and final tribes both receive positive values by the encoding in~\eqref{eq:y_config}. Considering $x$ as a configuration on the cycle with a sign-reversing edge, and using the notion of tribes from Section~\ref{sec:signed}, the number of tribes in $x$ again coincides with the number of tribes in $y$.

\begin{claim}\label{claim:odd}
Suppose that $\tau(y)\ge3$ is odd. Then there exists a coupling between the competing urn process started in $y$ and the auto-annihilative process started in $x$ such that, with probability one, at each time $t\ge0$ either both processes have $\tau(y)$ tribes remaining or neither does.
\end{claim}

\begin{proof}[Proof of claim]
Couple the two processes so that they evolve according to the same clocks. Then, as long as no tribe becomes extinct, each nucleation has the same effect in both processes.
\end{proof}

From Theorem~\ref{thm:signed} it follows that, with probability one, the auto-annihilative process $(Z(t))_{t\ge0}$ started in $x$ satisfies
$$
\lim_{t\to\infty}e^{-\lambda^* t}Z(t)=\big(R\cos\big(\pi(j-1)/N+S\big)\big)_{j=1}^N,
$$
for some continuous random variables $R$ and $S$.

Since the cosine function has precisely one root on any interval of the form $[\theta,\theta+\pi)$, it follows that, for large $t$, the auto-annihilative process consists of a single tribe. Together with Claim~\ref{claim:odd} we conclude that, with probability one, the competing urn scheme starting in $y$ will eventually consist of at most $\tau(y)-1$ tribes.

To complete the proof of Theorem~\ref{thm:urn} we now notice that, regardless of the number of tribes present in the initial configuration $y$, one of the tribes will eventually become extinct. Let $T_1$ denote the time at which the first tribe goes extinct. Restarting the argument with the configuration at time $T_1$ as the initial configuration, and appealing to the Markov character of the process, we conclude that at some future time $T_2$ another tribe will eventually have to become extinct, almost surely. Since the number of tribes can only decrease, and since there are no more than $N$ tribes to start with, we conclude that almost surely there will eventually be a single surviving tribe, and Theorem~\ref{thm:urn} follows.

\section{Non-coexistence for a spatial model of competition}\label{sec:plane}

Finally, we explain the connection between the $K$-type competing growth model on $\Z^2$ and the $K$-type competing urn scheme on a cycle, and deduce Theorem~\ref{thm:plane} from Theorem~\ref{thm:urn}. In particular, we show that the competition process on $\Z^2$, for any number of types, will almost surely have a single surviving type. The deduction is similar to that for $K=2$ types in~\cite{ahlgrijanmor19}.

Consider any configuration $z\in\{0,1,2,\ldots,K\}^{\Z^2}$ of $K$ types assigned to the vertices of $\Z^2$ consisting of a finite number of occupied sites (i.e.\ sites with values $1,2,\ldots,K$) and such that every vacant site (i.e.\ site with value 0) has at most one neighbour of each different type. The set of occupied sites in $z$ (i.e.\ its nonzero elements) can be identified with the bounded region $A(z)$ of $\R^2$ obtained by centring a unit square at each occupied vertex. We shall write $(X(t))_{t\ge0}$ for the $K$-type competing growth process on $\Z^2$ evolving from the initial configuration $z$, and let $A_t:=A(X(t))$.


Let $T_c$ denote the first time at which $X(t)$ is connected as a subset of $\Z^2$, and hence that $A_t$ is a connected region in $\R^2$, which is almost surely finite.
For $t\ge T_c$ the outer boundary of $A_t$ consists of piecewise linear segments. We shall refer to any maximal piecewise linear segment in the outer boundary of $A_t$ corresponding to a single type in $X(t)$ as a \emph{tribe}. For $t\ge T_c$, unless one of the $K$ types encoded in $X(t)$ already surrounds the others (in which case there is nothing to prove), the perimeter of $A_t$ will consist of some number $k\ge2$ tribes. Due to the Markov character of the process, and as time evolves the number of tribes can only decrease, it will suffice to prove that the number of tribes will decrease in finite time, almost surely.

We may encode the outer boundary of $A_t$ as follows: Pick a point on the perimeter that marks a break between two tribes. Follow the perimeter counterclockwise from the selected point until reaching a corner or another break between two tribes, whichever comes first. If this part of the perimeter corresponds to type $j$ and has length $\ell$, then set $U_1(t)=\ell e_j$, where $e_j$ is the $j$th coordinate vector in $\R^K$. If the stop did not occur at a corner, then also set $U_2(t)$ to be the $K$-dimensional zero-vector. Suppose next that $U_i(t)$ has been defined. Repeat the procedure to define $U_{i+1}(t)$, or both $U_{i+1}(t)$ and $U_{i+2}(t)$, depending on whether the stop occurs at a corner or not. Stop this process once one full lap around the perimeter is completed.

\begin{claim}\label{claim:perimeter}
Suppose that $t\ge T_c$ and that the perimeter of $A_t$ consists of $k\ge2$ tribes. Then the encoding of the perimeter results in precisely $2k+4$ vectors $U_i(t)$ being defined.
\end{claim}

\begin{proof}[Proof of claim]
For $t\ge T_c$ the perimeter of $A_t$ consists of $m$ inside corners and $m+4$ outside corners, for some $m\ge0$. Note that an inside corner must correspond to a break between tribes, but that an outside corner cannot. Hence, $m$ breaks occur at inside corners and the remaining $k-m$ breaks occur outside of the corners (i.e.\ on line segments). Since each corner (either inside or outside) results in one $U_i$-vector being defined, and each break on a line segment results in two $U_i$-vectors being defined, this means that $(2m+4)+2(k-m)=2k+4$ vectors are defined in total.
\end{proof}

The key step of the proof is the next, which uses the $U_i$-encoding to construct a coupling the growth process on $\Z^2$ and that on a cycle. Suppose that $z$ is connected as a subset of $\Z^2$ and let $y=(y_1,y_2,\ldots,y_N)$ be the element of $\{0,1,\ldots\}^{\{1,2,\ldots,N\}\times\{1,2,\ldots,K\}}$ obtained from the $U_i$-encoding of $z$.

\begin{claim}\label{claim:U-coupling}
Suppose that $z$ is connected and consists of $k\ge2$ tribes. Then there exists a coupling between the $K$-type competing growth process on $\Z^2$ started in $z$ and the $K$-type competing urn scheme on a cycle of length $2k+4$ started in $y$ such that, with probability one, at each time $t\ge0$ either both processes have $k$ tribes remaining or neither does.
\end{claim}

\begin{proof}[Proof of claim]
Let $U(t)=(U_1(t),U_2(t),\ldots,U_N(t))$ be obtained from the encoding of the perimeter of $A_t$. Let $T_0$ denote the first time that a tribe goes extinct. Consider a process that evolves according to $U(t)$ for $t<T_0$ and independently from $U(t)$, following the the dynamics of the competing urn scheme, for $t\ge T_0$. By Claim~\ref{claim:perimeter}, this defines $2k+4$ vectors $U_i(t)$ for each $t<T_0$. It is straightforward to verify that this process is equal in distribution to the competing urn scheme on a cycle of length $2k+4$, for all $t\ge0$.
\end{proof}

To complete the proof of Theorem~\ref{thm:plane}, let the initial configuration $z\in\{0,1,\ldots,K\}^{\Z^2}$ be arbitrary among all configurations with a finite number of nonzero coordinates. 
By Theorem~\ref{thm:urn} we know that the competing urn scheme on any cycle will eventually consist of a single survivor, almost surely. In particular, the number of tribes will decrease in finite time with probability one. Due to the Markov character of the processes, and Claim~\ref{claim:U-coupling}, also in the competing growth process on $\Z^2$ the number of tribes will decrease in finite time, as required.

\section{Open problems}\label{sec:open}

We end this text with a few open problems that we hope will inspire future work.

For any finite connected graph $G$, in~\cite{ahlgrijanmor19} the authors proved that the two-type competing urn scheme has a single surviving colour, almost surely. Moreover, the authors gave examples of finite connected graphs for which coexistence of $K\ge3$ types is possible (with positive probability). The path and the cycle are the only known examples of finite connected graphs for which the $K$-type competing urn scheme almost surely has a single surviving type also for $K\ge3$. In~\cite{ahlgrijanmor19} the authors posed the problem of characterising, for each $K\ge3$, the family of graphs for which the competing urn scheme has a single survivor.

The examples in~\cite{ahlgrijanmor19} where coexistence occurs have the property that there exists a vertex which, if removed, disconnects the graph into $K\ge3$ identical subgraphs. It is possible to imagine generalisations of the same idea, where the removal of a small number of vertices disconnects the graph into identical pieces, and where coexistence occurs with positive probability. However, for graphs with a more regular structure, e.g.\ vertex-transitive graphs, it seems reasonable to expect that coexistence cannot happen. In particular, we conjecture this to be the case for the hypercube and the two- or higher-dimensional discrete torus.

\begin{conj}
For every $K\ge3$, and any nonzero initial configuration, the $K$-type competing urn scheme on the hypercube and the discrete torus has almost surely a single surviving type.
\end{conj}

We have in this paper considered a setting where balls of different types nucleate according to Poisson clocks that tic at the same rate. More generally, one could allow clocks of balls of different type to tic at different rate. Our methods to not seem to extend easily to cover this more general setting. In particular, it seems challenging to encode the competition between different types in a system consisting of positive and negative particles in a setting with more than two rate parameters. We believe that coexistence cannot happen also in this setting, but we expect that new ideas will be required to verify this belief.

\begin{conj}
For every $K\ge3$, and any nonzero initial configuration, the $K$-type competing urn scheme on the cycle, where balls of type $j$ have clocks that tic at rate $\lambda_j$, has almost surely a single surviving type.
\end{conj}

In~\cite{ahlgrijan21}, the competing urn scheme was studied on the $\Z^d$ nearest-neighbour lattice, for $d\ge1$. The authors considered the infinite initial configuration consisting of one ball at each position, whose colour (either red or blue) was determined by i.i.d.\ coin flips. The configuration of balls induces a colouring of $\Z^d$, where the colour of each site is given by the colour of the balls currently present at the site (and white, say, if no ball is currently present). The authors studied the question of fixation, and showed that when the coin is fair, every site changes colour infinitely often almost surely, whereas for a biased coin, the colouring fixates locally (to the dominating colour) in the sense that for each site there exists an almost surely finite time after which the colour no longer changes. For $K\ge3$ the initial colouring could be determined by the roll of a (possibly biased) $K$-sided die. The analogous problem in this setting remains open.

\begin{problem}
For $K\ge3$ and $d\ge1$, determine under what conditions the $K$-type competing urn scheme on $\Z^d$ fixates and not.
\end{problem}

The two-type competing urn scheme on $\Z$, evolving from a uniform random initial colouring, can be described as the evolution of monochromatic intervals competing for space. In~\cite{ahlgrijan21}, the authors found that the height (i.e.\ number of balls) of a typical interval is $\Theta(t^{-1/4}e^{2t})$ at arbitrarily large times $t$, and posed the problem of determining the width of a typical interval. It is tempting to look towards the auto-annihilative growth process on a cycle when trying to understand the evolution of the monochromatic intervals. Indeed, the auto-annihilative process is indicative of the anatomy of an isolated interval, when evolving in equilibrium.
Equating the typical height of an interval with the growth rate $e^{\lambda^* t}$ of the auto-annihilative process suggests that the width of a typical interval ought to grow as $\Theta(\sqrt{t/\log(t)})$.

\begin{problem}
For the two-type competing urn scheme on $\Z$, determine the typical width of a monochromatic interval. Is it true that the width is $\Theta(\sqrt{t/\log(t)})$ with high probability?
\end{problem}

As an alternative to the random initial configuration considered in~\cite{ahlgrijan21}, one could consider the setting with $K\ge3$ types initially present, or where all particles initially present (one for each site) are of distinct types. For $d=1$, the above question regarding the length of a typical interval remains relevant.

\noindent
{\bf Acknowledgements.} Research in part supported by the Swedish Research Council through grant 2021-03964 (DA) and 2016-04566 (CF).


\newcommand{\noopsort}[1]{}\def\cprime{$'$}

\bigskip

{\small
\noindent
{\sc Department of Mathematics, Stockholm University\\ 
SE-10691 Stockholm, Sweden}\\
\texttt{daniel.ahlberg@math.su.se}, \texttt{carolina.fransson@math.su.se}\\

\end{document}